\documentclass[11pt,reqno]{amsart}
\usepackage[letterpaper,margin=1in]{geometry}
\usepackage{amsmath,amssymb,amsthm,mathtools}
\usepackage{booktabs,array,longtable,enumitem}
\usepackage{microtype}
\usepackage{hyperref}
\usepackage{bookmark}
\usepackage{needspace}
\newcommand{\Z}{\mathbb Z}
\newcommand{\Q}{\mathbb Q}
\newcommand{\F}{\mathbb F}
\newcommand{\OO}{\mathcal O}
\newcommand{\C}[1]{C_{#1}}
\newcommand{\Sdet}[1]{S(#1)}
\DeclareMathOperator{\Norm}{N}

\theoremstyle{plain}
\newtheorem{theorem}{Theorem}[section]
\newtheorem{lemma}[theorem]{Lemma}
\newtheorem{proposition}[theorem]{Proposition}
\newtheorem{corollary}[theorem]{Corollary}
\theoremstyle{definition}

\theoremstyle{remark}

\numberwithin{equation}{section}
\setlist[itemize]{topsep=3pt,itemsep=1pt}
\setlist[enumerate]{topsep=3pt,itemsep=2pt}
\hypersetup{colorlinks=true,linkcolor=blue,citecolor=blue,urlcolor=blue,
 pdftitle={Integer group determinants for abelian groups of order 24},
 pdfauthor={Chatchawan Panraksa}}
\title[Integer group determinants of order 24]{Integer group determinants\\for abelian groups of order 24}
\author{Chatchawan Panraksa}
\address{Applied Mathematics Program, Mahidol University International College,
Salaya, Nakhon Pathom 73170, Thailand}
\email{chatchawan.pan@mahidol.ac.th}
\date{9 September 2026}
\subjclass[2020]{11C20, 11R18, 15B36, 20C05}
\keywords{Integer group determinant, circulant determinant, integral group ring,
cyclotomic norm, conductor, finite abelian group}
\begin{document}
\begin{abstract}
We determine the integer group determinant sets for
$C_{24}$, $C_2\times C_{12}$ and $C_2^2\times C_6$, the three abelian
groups of order $24$. Character factorization expresses each determinant
as a product of cyclotomic norms, whose components must satisfy integral
compatibility conditions.
The classifications are constructive and reduce each nonzero signed
integer to valuation conditions at $2$ and $3$ and a finite test on its
remaining prime factors. The tests use conductor quotients and
norm-preserving global units. Each prime contribution has a uniform
bound independent of its multiplicity, and successful tests yield
integer coefficients realizing the prescribed determinant.
The critical cyclic strata share one compatibility group. For
$C_2^2\times C_6$, a normalized rational--Eisenstein pair settles the
odd values and most even strata; target stabilizers reduce the
remaining tests. We also prove that the odd determinant set of
$C_2^2\times C_6$ is properly contained in that of $C_2\times C_{12}$,
exhibit an infinite family in the difference, and determine the signed
powers of two in all three sets. Exact arithmetic
certificates and coefficient constructors accompany the proofs.
\end{abstract}
\maketitle

\section{Introduction}\label{sec:introduction}

For a finite group $G$, its integer group determinant set is
\[
 S(G)=\left\{\det(a_{gh^{-1}})_{g,h\in G}:a_g\in\Z\right\}.
\]
When $G=\C{n}$, these are the integer circulant determinants of order
$n$. Character factorization expresses an abelian group determinant as a
product of cyclotomic norms. The factors must satisfy congruences
because they come from one integral group-ring element.

Pinner and Smyth determine the sets
for groups of order at most fourteen and several further groups
\cite{PinnerSmyth2020}. Subsequent work treats the abelian groups of
order sixteen \cite{YamaguchiAbelian16,YamaguchiCirculant16}, three
nonabelian groups of that order \cite{YamaguchiNonabelian16}, and the
groups of order eighteen \cite{PaudelPinner18}. For $\C{5}\times\C{5}$,
Panraksa \cite{Panraksa2026} proves that the $5$-divisible determinant
values are exactly $5^8\Z$, completing the classification for this
group. Prime-power circulants have been studied in
\cite{MossinghoffPinner2023}. For mixed cyclic
orders, the order-fifteen classification already requires conditions
on the prime factors of the integer to be represented
\cite{PaudelPinner15}. Product factorizations and subgroup relations
provide further tools \cite{PaudelPinner2023,YamaguchiDedekind2024}.

We determine $S(\C{24})$, $S(\C{2}\times\C{12})$ and
$S(\C{2}^2\times\C{6})$ in Theorems~\ref{cy24:main},
\ref{c2c12:classification} and \ref{c2c2c6:classification}.
The cyclic classification includes the known families of integers prime
to $6$ and multiples of $288$, given by Kaiblinger's
Theorem~4.4(i)~\cite{Kaiblinger2012} with $q=3$; its critical strata
share one obstruction group. For $\C{2}^2\times\C{6}$, a normalized
rational--Eisenstein pair settles the odd values and most even strata.
The $27$ exceptional pairs use target stabilizers to reduce the tests;
four require at most one prime move.

For $D=2^k3^bm\ne0$, with $m$ signed and $\gcd(m,6)=1$,
Table~\ref{tab:classification-guide} gives sufficient thresholds for
every such $D$ to occur, provided $b\ne1$. Valuation $b=1$ is always excluded.

\begin{table}[tb]
\centering
\caption{Compatibility groups and classification routes.}
\label{tab:classification-guide}
\begin{tabular}{@{}lccp{3.1cm}@{}}
\toprule
$G$ & Compatibility group & Threshold & Normalized-pair test\\
\midrule
$\C{24}$ & $\C{2}^2\times\C{4}$ & $k\ge8$ & Not needed\\
$\C{2}\times\C{12}$ & $\C{2}^5\times\C{4}^2$ & $k\ge22$ & Not needed\\
$\C{2}^2\times\C{6}$ & $\C{2}^{10}\times\C{4}$ & $k\ge32$ &
Odd and selected even strata\\
\bottomrule
\end{tabular}
\end{table}

The common argument allocates prime factors among character components
and uses norm-preserving global units to restore integrality.
Theorem~\ref{thm:prime-allocation} bounds each prime contribution
independently of its multiplicity, giving constructive criteria for
arbitrary signed integers. The group-specific work determines the
compatibility quotients and their local class sets.

The classifications also distinguish the two noncyclic groups in
ways that valuation restrictions alone do not detect. We prove
\[
 S_{\mathrm{odd}}(\C{2}\times\C{2}\times\C{6})
 \subsetneq S_{\mathrm{odd}}(\C{2}\times\C{12}),
\]
where $S_{\mathrm{odd}}$ denotes the odd members of a determinant set.
An explicit infinite family in the difference is
$3^4\cdot5\cdot13\cdot73^{r+1}$, $r\ge0$.
We also give the exact signed powers of two in all three sets.

Section~\ref{sec:framework} gives the framework, Sections~3--5 the
classifications, and Section~\ref{sec:comparison} their comparisons.
The appendices specify the finite data; Section~\ref{sec:computation}
gives the computational methodology and reproduction guide.


\section{Integral compatibility and prime allocation}\label{sec:framework}

Let $G$ be a finite abelian group and put $A=\Z[G]$. Multiplication by
$F=\sum_{g\in G}a_g g$ on the group basis has matrix
$(a_{gh^{-1}})_{g,h\in G}$. Its determinant is therefore the algebra norm
of $F$. Evaluating at the complex characters diagonalizes this matrix.
Grouping Galois-conjugate characters gives
\begin{equation}\label{eq:character-norm}
 \Q[G]\simeq\prod_j K_j,\qquad
 D_G(F)=\prod_j\Norm_{K_j/\Q}(F_j).
\end{equation}
Each $K_j$ is cyclotomic. Throughout, the norm of a rational component
is that component itself, with its sign. A nonzero norm from a totally
imaginary field is positive.

The character fields used here are $\Q$ and $K_d=\Q(\zeta_d)$, where $d=3,4,6,8,12,24$.
Write $O_d=\Z[\zeta_d]$. We use the following arithmetic input.

\begin{lemma}\label{lem:fields}
Every $O_d$ just listed has class number one. In each row of
Table~\ref{tab:field-units}, the displayed elements generate its full unit
group. All these units have absolute norm one. For a prime $q>3$, every
prime ideal above $q$ in any of these fields has residue degree one or two.
\end{lemma}

\begin{table}[tb]
\centering
\caption{Full field units; in each row $z=\zeta_d$.}\label{tab:field-units}
\begin{tabular}{crl}
\toprule
$d$ & $\operatorname{disc}(K_d)$ & Generators of $O_d^\times$\\
\midrule
3 & $-3$ & $1+z$\\
4 & $-4$ & $z$\\
6 & $-3$ & $z$\\
8 & $256$ & $z,\ 1+z+z^2$\\
12 & $144$ & $z,\ z-1$\\
24 & $2^{16}3^4$ & $z,\ 1+z^2-z^6,\ z-1,\ -1-z^5$\\
\bottomrule
\end{tabular}
\end{table}

\begin{proof}
The residue degree is the order of $q$ modulo $d$. Every unit modulo
$24$ has square one, proving the last assertion. The discriminants
follow from the cyclotomic discriminant formula.

Masley's classification~\cite[Main Theorem, p.~179]{Masley1976} gives
class number one, with $K_6=K_3$. The displayed generators, expressed
in the certified full unit group,
generate a coordinate lattice of index one together with the torsion
relation. Their norms are checked in the defining polynomial quotient.
The exact certificates and regeneration route are described in
Section~\ref{sec:computation}.
\end{proof}

\paragraph{The \texorpdfstring{$\C{3}$}{C3} factor.}
All three groups have the form $P\times\C{3}$ with $|P|=8$.
Put $R=\Z[P]$, $O=\Z[\omega]$, with $\omega^2+\omega+1=0$,
and $\lambda=1-\omega$.
Evaluation at $z=1,\omega$, with $b=b_0+b_1\omega$, gives
\begin{gather}
 \Z[P\times\C{3}]\simeq
 \{(a,b)\in R\times O[P]:a\equiv b\pmod\lambda\},
 \label{eq:c3-evaluation}\\
 q=(a-b_0-b_1)/3,\qquad
 F=(b_0+q)+(b_1+q)z+qz^2,\label{eq:c3-inverse}\\
 D_{P\times\C{3}}(F)=D_P(a)\,
 \Norm_{\Q(\omega)/\Q}\bigl(D_P(b)\bigr).
 \label{eq:c3-norm}
\end{gather}
The coefficientwise congruence makes $q\in R$; evaluation verifies the
inverse and its uniqueness. Character factorization gives the norm identity.

For any abelian $P$ of order $8$, a pair
$A_0=1+8a$, $\beta=1+8b+8c\omega$, where $a,b,c\in\Z$ and
$A_0\equiv\beta\pmod\lambda$, has the realization
\begin{equation}\label{eq:normalized-pair-construction}
 \begin{gathered}
 F=1+J_P(r+sz+tz^2),\qquad J_P=\sum_{g\in P}g,\\
 t=(a-b-c)/3,\qquad r=b+t,\qquad s=c+t.
 \end{gathered}
\end{equation}
Indeed, the congruence makes $r,s,t$ integral. The trivial $P$-character
gives $A_0,\beta,\overline\beta$, and all other characters give one.
Thus $D_{P\times\C{3}}(F)=A_0\Norm(\beta)$, with its given sign.

\subsection{The finite compatibility condition}

Let $M=\prod_j O_j$ be the maximal order in the algebra
\eqref{eq:character-norm}, with $O_j=\Z$ for a rational component.
The conductor $\mathfrak f=\{x\in M:xM\subseteq A\}$ and the order
index follow from the uniform formulas, with $N=|G|$,
\begin{equation}\label{eq:abelian-conductor}
 \mathfrak f_j=N\mathfrak D_{K_j/\Q}^{-1},\qquad
 [M:A]^2=\frac{N^N}{\prod_j|\operatorname{disc}(K_j)|}.
\end{equation}
An isolated component $a_j$ has coefficient
$N^{-1}\operatorname{Tr}_{K_j/\Q}(a_j\chi_j(g^{-1}))$ at $g$.
The character values run through all powers of the primitive root defining
$K_j$, and hence span $O_j$. Its integral-coefficient locus is therefore
$N\mathfrak D_{K_j/\Q}^{-1}$. This is an $O_j$-ideal contained in
$O_j$, since its elements are evaluations of elements of $A$;
it is the conductor component. The regular trace matrix on the group
basis is $N$ times the inversion permutation matrix, giving absolute
discriminant $N^N$ and the index formula.

For our groups the conductor is supported at $2$ and $3$.
Since it is an ideal of $M$ contained in $A$, a tuple belongs to $A$
exactly when its residue belongs to $A/\mathfrak f$.

Set $T=(M/\mathfrak f)^\times$. Let $L$ be the units of the image
of $A$ in $M/\mathfrak f$, and let $E_+$ be the image in $T$ of
$U_+=\{\varepsilon\in M^\times:\prod_j\Norm(\varepsilon_j)=1\}$.
Multiplication by an element of $U_+$ is a \emph{norm-preserving unit
correction}. The finite abelian group $H=T/(E_+L)$ measures the
obstruction to such a correction into $A$.
We write $h:T\longrightarrow H$ additively.

\begin{lemma}\label{lem:unit-correction}
If every component norm of $v\in M$ is prime to six, there is an
$\varepsilon\in U_+$ such that $\varepsilon v\in A$ if and only if
$h(v)=0$. For the same $v$, fix $r\in M$ with nonzero components and
put
\[
 Y_r=\{t\in T:rt\bmod\mathfrak f\in A/\mathfrak f\}.
\]
Then some $\varepsilon\in U_+$ satisfies $r\varepsilon v\in A$
if and only if $h(v)\in h(Y_r)$.
\end{lemma}

\begin{proof}
The image of $A$ is a finite ring, so its elements which are units
in $M/\mathfrak f$ form its unit group $L$: the inverse is a
positive power of the element in this finite ring. Thus the first
condition is exactly $v\bmod\mathfrak f\in E_+L$.
The set $Y_r$ is stable under multiplication by $L$ and its inverses.
Consequently
$h^{-1}(h(Y_r))=E_+Y_r$. Lifting the $E_+$ factor to an actual
global unit proves sufficiency; necessity follows by reduction.
Retaining generator words supplies the correcting global unit; the
inverse character map recovers the integral coefficients.
\end{proof}

\paragraph{Chinese remainder assembly.}
By Chinese remainders, write $M/\mathfrak f=M_2\times M_3$ and
$A/\mathfrak f=A_2\times A_3$, with $T_p=M_p^\times$.
For raw tuples $r_2,r_3$ supported above $2,3$, respectively, denote the
image of $r_p$ at $q$ by $r_{p,q}$ and put
\[
 \begin{aligned}
 Y_2&=\{u\in T_2:r_{2,2}u\in A_2\},&
 B_2&=h\bigl(Y_2\times\{r_{2,3}^{-1}\}\bigr),\\
 Y_3&=\{v\in T_3:r_{3,3}v\in A_3\},&
 B_3&=h\bigl(\{r_{3,2}^{-1}\}\times Y_3\bigr).
 \end{aligned}
\]
The local conditions give
$Y_{r_2r_3}=(r_{3,2}^{-1}Y_2)\times(r_{2,3}^{-1}Y_3)$; applying $h$ yields
\begin{equation}\label{eq:crt-assembly}
 h(Y_{r_2r_3})=B_2+B_3.
\end{equation}
Each inverse is taken only at the opposite prime, where the raw factor
is a unit. Lemma~\ref{lem:unit-correction} applies this identity to the
supplied cofactor tuple, retaining its component absolute norms and total sign.

\paragraph{Finite unit generators.}
For a finite commutative ring $S$ and a nilpotent ideal $I$, multiplication
gives, for $j\ge1$,
\begin{equation}\label{eq:unit-filtration}
 (1+I^j)/(1+I^{j+1})\simeq I^j/I^{j+1},\qquad
 1+x\longmapsto x.
\end{equation}
Indeed, products become sums modulo $I^{j+1}$. Every lift of a unit of
$S/I$ is a unit: lift an inverse and invert the resulting element of
$1+I$ by a finite geometric series. Lifts of generators of $(S/I)^\times$,
together with $1+x$ for additive generators of each layer, therefore
generate $S^\times$, by successive cancellation. The appendices specify
the ideals, nilpotence bounds, and spanning elements.

\subsection{A prime-factor test with a uniform finite bound}

In this subsection, $M=\Z\times\prod_{j=1}^s O_j$, where each $O_j$
is either $\Z$ or the ring of integers of a totally imaginary
class-number-one field. Choose the first rational component as position
zero. Suppose $h$ is a homomorphism from the tuples prime to six to a
finite abelian group $H$, obtained from finite residue rings and killing
the image of $U_+$. Write $h_j$ for its component maps and put
$c(m)=h_0(m)$ for signed $m$ prime to six.

For each prime $q>3$ and each position $j>0$, factor $qO_j$.
For every prime ideal $\mathfrak q\mid qO_j$ of norm $q^f$,
choose an integral generator $\alpha_{\mathfrak q}$ and attach the
move
\begin{equation}\label{eq:prime-move}
 h_j(\alpha_{\mathfrak q})-f h_0(q),\qquad\text{of cost }f.
\end{equation}
For a rational position use the generator $q$ and cost one.
Include every prime ideal and component position.
Let $R_q(e)$ consist of the sums of these moves having total cost at
most $e$, including the empty sum. Repetition of a move is allowed.
The subtraction records removal of the same norm from position zero.

\begin{theorem}\label{thm:prime-allocation}
For every nonzero signed integer $m$ prime to six, the set of classes
$h(v)$ of integral tuples with signed norm product $m$ is exactly
\begin{equation}\label{eq:prime-allocation}
 c(m)+\sum_{q^e\parallel |m|}R_q(e).
\end{equation}
Here the sum is a sumset in $H$, and the empty sum is $\{0\}$.
If all costs are at most $w$, each $R_q(e)$ is unchanged when $e$
is replaced by $\min\{e,w(|H|-1)\}$.
The description is independent of the chosen prime generators and
constructs a tuple for each class in \eqref{eq:prime-allocation}.
\end{theorem}

\begin{proof}
Factor each nonrational component ideal of $v$ into prime ideals.
Class number one writes that component as a product of the chosen
generators times a global unit. Factor the absolute values of the
rational components as well. For a fixed $q$, their combined norm
contribution is at most $q^e$; unused powers lie in position zero.
All nonrational units have norm one. The rational signs left after
putting the total sign in position zero have product one, so the
remaining unit tuple is in $U_+$. Applying $h$ proves necessity.

Conversely, select move words realizing one element of each $R_q(e)$.
Put the selected prime generators in their indicated components,
and put $m$ divided by their positive norm product in position zero.
The cost inequalities make this last quantity a signed integer.
The resulting tuple has norm product $m$, and its class is exactly
the selected sum. A different generator of the same prime ideal
differs by a global unit, which $h$ kills in a nonrational position.
Positive generators fix this choice in rational positions.

Finally, let $\mathsf D(H)$ be the Davenport constant, the least
length forcing a nonempty zero-sum subsequence. A minimum-cost word
has no such subsequence, since deletion preserves its endpoint and
reduces cost. It therefore costs at most $w(\mathsf D(H)-1)$.
The $|H|+1$ partial sums, including zero, of any $|H|$ terms have a
repetition, so $\mathsf D(H)\le|H|$, proving the bound.
\end{proof}

For $H\simeq\prod_i\C{n_i}$ with each $n_i$ a power of $2$,
Olson's $p$-group theorem~\cite{Olson1969} gives
$\mathsf D(H)=1+s$, where $s=\sum_i(n_i-1)$. Thus
\begin{equation}\label{eq:two-group-budget}
 R_q(e)=R_q\bigl(\min\{e,ws\}\bigr).
\end{equation}
An elementary proof is retained in the supplementary background note.

For the obstruction groups
$\C{2}^2\times\C{4}$, $\C{2}^5\times\C{4}^2$ and
$\C{2}^{10}\times\C{4}$ that occur below, the values of $s$ are
$5$, $11$ and $13$, respectively.  With residue degrees at most two,
the corresponding prime budgets are therefore $10$, $22$ and $26$.
These bounds concern prime allocation and do not change the local
valuation thresholds.

For a prescribed compatibility set $B$, test whether $B$ meets
\eqref{eq:prime-allocation}. Selected prime words and
Lemma~\ref{lem:unit-correction} recover integral coefficients.
When all rational component maps vanish, $c(m)$ and the subtractions
in \eqref{eq:prime-move} vanish; this applies to the cyclic and
small Eisenstein tests.


\section{The cyclic group of order 24}\label{cy24:section}

Four critical valuation strata share an obstruction group of order $16$;
the other strata admit explicit parameter families.

Put
\[
 \mathcal D=\{1,2,3,4,6,8,12,24\},\qquad
 K_d=\Q(\zeta_d),\qquad \OO_d=\Z[\zeta_d],
\]
with $K_1=K_2=\Q$ and $\OO_1=\OO_2=\Z$. All six nonrational
fields have class number one, by Lemma~\ref{lem:fields}.
Write $H=\C{2}\times\C{2}\times\C{4}$ additively, with coordinates
modulo $2,2,4$. For $d\in\mathcal D$, let
\[
 \psi_d:\{\alpha\in\OO_d: (\Norm_{K_d/\Q}(\alpha),6)=1\}
       \longrightarrow H
\]
be the component homomorphism defined by the finite residue table in
Appendix~\ref{cy24:residue-data}. Its definition uses only multiplication
in explicitly specified rings of characteristic $2$, $3$, $4$ or $8$.
Every global unit is killed, and $\psi_1=\psi_2=0$.

Apply Theorem~\ref{thm:prime-allocation} with component maps $\psi_d$
and distinguished component $d=1$. Since $\psi_1=\psi_2=0$, the
moves reduce to $\psi_d(\alpha_{\mathfrak q})$, for $q>3$ and
$\mathfrak q\mid q$ in $K_d$, $d>2$, at cost $f$ when
$\Norm(\mathfrak q)=q^f$. Every such cost is at most two, since
$(\Z/24\Z)^\times$ has exponent two.
Let $\ell_q(h)$ be the minimum cost of a prime word representing $h$,
or infinity if none exists, with the empty word of cost zero.
For $m>0$ prime to six, define
\begin{equation}\label{cy24:prime-test}
 \mathcal R(m)=\sum_{q^e\parallel m}
       \{h\in H:\ell_q(h)\le\min(e,10)\},
 \qquad \mathcal R(1)=\{0\}.
\end{equation}
Thus only sixteen states are retained at each prime. The bound ten
follows from~\eqref{eq:two-group-budget}, since
$H=\C{2}^2\times\C{4}$ and the move costs are at most two.

Set
\begin{equation}\label{cy24:targets}
\begin{split}
 c&=(0,1,2),\qquad B=\{(0,1,0),(0,1,2)\},\\
 B'&=B\cup\{(1,0,1),(1,0,3),(1,1,1),(1,1,3)\}.
\end{split}
\end{equation}

\Needspace{18\baselineskip}
\begin{theorem}\label{cy24:main}
Zero belongs to $\Sdet{\C{24}}$. For $D\ne0$, write
$|D|=2^a3^bm$, where $m>0$ and $(m,6)=1$. Membership is independent
of the sign and is given by the following exhaustive table.
\begin{center}
\begin{tabular}{ll}
\toprule
Valuations & Necessary and sufficient condition\\
\midrule
$a=b=0$ & Always attained\\
$1\le a\le4$, or $b=1$ & Never attained\\
$a\ge5$, $b\ge2$ & Always attained\\
$a=0$, $b\ge3$ & Always attained\\
$a=0$, $b=2$ & $c\in\mathcal R(m)$\\
$a\in\{5,6\}$, $b=0$ & $\mathcal R(m)\cap B\ne\varnothing$\\
$a=7$, $b=0$ & $\mathcal R(m)\cap B'\ne\varnothing$\\
$a\ge8$, $b=0$ & Always attained\\
\bottomrule
\end{tabular}
\end{center}
Every admitted value has a constructive realization with $24$
integer coefficients.
\end{theorem}

In particular, the cofactor sets at $2$-adic valuations $5$ and $6$ coincide
and are contained in the cofactor set at valuation $7$. The same test
also shows that $9m\in\Sdet{\C{24}}$, for $(m,6)=1$, implies
$32m,64m\in\Sdet{\C{24}}$. The pure values $32$, $64$ and $128$
are excluded, whereas every multiple of $256$ prime to three occurs.

\subsection{Integral components and their critical supports}

Let $A=\Z[X]/(X^{24}-1)$ and $M=\prod_{d\in\mathcal D}\OO_d$.
Evaluation embeds $A$ in $M$, and
\[
 D_{24}(F):=\prod_{z^{24}=1}F(z)
       =\prod_{d\in\mathcal D}\Norm_{K_d/\Q}(F(\zeta_d)).
\]
The two rational factors carry their signs. Multiplication of $F$ by $X$
reverses $D_{24}(F)$, since the product of the $24$ roots is $-1$.
Replacing $F(X)$ by $F(-X)$ preserves the determinant.

The conductor $J=(A:M)$ has component ideals
\begin{equation}\label{cy24:conductor}
\begin{split}
 (J_1,J_2,J_3,J_4)&=(24\Z,24\Z,8\lambda_3\OO_3,12\OO_4),\\
 (J_6,J_8,J_{12},J_{24})&=(8\lambda_6\OO_6,6\OO_8,
                         4\lambda_{12}\OO_{12},2\lambda_{24}\OO_{24}),
\end{split}
\end{equation}
where $\lambda_d=1-\zeta_d^{d/3}$ when $3\mid d$.
Formula~\eqref{eq:abelian-conductor}, with the differents obtained
from $\Phi_d'$, gives these ideals and $[M:A]=2^{21}3^8$.

Here is an explicit inverse that will also construct coefficients.
Put $x=X^{16}$, $y=X^9$, so that $x^3=y^8=1$ and $X=xy$.
For $e\in\{1,2,4,8\}$ identify
$\zeta_{3e}$ with $\omega\xi_e$, where $\omega^2+\omega+1=0$
and $\xi_e$ is the specified primitive $e$th root. For a tuple
$(u_1,u_2,u_4,u_8)$, write $u_4=c_0+c_1\xi_4$ and
$u_8=\sum_{j=0}^3d_j\xi_8^j$, and put
\begin{equation}\label{cy24:T8}
\begin{split}
 t&=(u_1+u_2+2c_0,\ u_1-u_2+2c_1,
       u_1+u_2-2c_0,\ u_1-u_2-2c_1),\\
 T_8(u_1,u_2,u_4,u_8)&=(t+4d,t-4d).
\end{split}
\end{equation}
This vector is eight times the coefficient vector of the interpolant;
the formula works over $\Z$ and over $\Z[\omega]$.

Given components $\alpha_d$, expand
$\alpha_{3e}(\omega\xi_e)=u_e(\xi_e)+\omega v_e(\xi_e)$
in the integral tensor basis. Apply $T_8$ separately to
$(\alpha_e)$, $(u_e)$ and $(v_e)$, obtaining $P_*,Q_{0*},Q_{1*}$.
The tuple comes from $A$ if and only if
\begin{equation}\label{cy24:gluing}
 P_*,Q_{0*},Q_{1*}\equiv0\pmod8,
 \qquad P_*-Q_{0*}-Q_{1*}\equiv0\pmod{24}.
\end{equation}
Indeed, division by eight gives the $\C{8}$ coefficient vectors;
\eqref{eq:c3-evaluation} then gives the remaining congruence.
Formula~\eqref{eq:c3-inverse}, followed by $x=X^{16}$, $y=X^9$
and reduction modulo $X^{24}-1$, recovers the original coefficients.

\begin{lemma}\label{cy24:support}
For a nonzero determinant the following assertions hold.
\begin{enumerate}
\item Its $3$-adic valuation is $0$ or at least $2$. If it is exactly $2$,
the occupied norm components are $(1,3)$ or $(2,6)$, with valuation one
in each. The substitution $F(X)\mapsto F(-X)$ interchanges these supports.
\item Its $2$-adic valuation is $0$ or at least $5$. If it is $k\in\{5,6,7\}$,
all components with $3\mid d$ are units above two. The remaining
component norm valuations $(r_1,r_2,r_4,r_8)$ are positive, have sum $k$,
and satisfy $r_1+r_2+r_4\ge4$.
\end{enumerate}
\end{lemma}
\begin{proof}
At three, \eqref{cy24:gluing} pairs the $e$ and $3e$ components over
$\OO_e/3$, for $e=1,2,4,8$. These residue rings are
$\F_3,\F_3,\F_9,\F_9\times\F_9$. An occupied branch contributes
at least its residue degree to each of its two norm valuations.
This excludes valuation one and leaves precisely the two stated
linear branches at valuation two.

For the second assertion, use \eqref{eq:c3-norm} for the scalar and
Eisenstein evaluations $P\in\Z[\C{8}]$ and
$Q\in\Z[\omega][\C{8}]$. All eighth roots
reduce to one above two. If $Q(1)$ is a nonunit, its factors at $1,-1$
are divisible by two, and its factors from the extensions adjoining
$i$ and $\zeta_8$ have positive relative norm valuation. Therefore
$v_2(D_8(Q))\ge4$ over the unramified quadratic base, and its absolute
norm contributes at least eight. Otherwise all four factors are units.

For integer $P$, a nonunit likewise makes all four norm valuations
positive. The first three have sum at least four. When
$v_2(\Norm(P(i)))=1$, the real and imaginary parts of $P(i)$ are
both odd; hence $(P(1)-P(-1))/2$ is odd, and one rational factor is
divisible by four. If its valuation is at least two, the same product
bound follows immediately. Thus $v_2(D_8(P))\ge5$ and the asserted
profiles exhaust the three low strata.
\end{proof}

\subsection{One obstruction for the four critical strata}

For a tuple $\beta=(\beta_d)$ of integral cofactors whose norms are
prime to six, put $h(\beta)=\sum_d\psi_d(\beta_d)$. For a positive
profile $r=(r_1,r_2,r_4,r_8)$ define
\[
 S_r=(2^{r_1},2^{r_2},1,(1-i)^{r_4},1,
                         (1-\zeta_8)^{r_8},1,1),
 \qquad S_3=(3,1,1-\zeta_3,1,1,1,1,1),
\]
in the order $\mathcal D$.

\begin{proposition}\label{cy24:cofactor}
Independent global component units make $S_3\beta$ integral if and
only if $h(\beta)=c$. For $k=5,6,7$, they make $S_r\beta$ integral
for some positive profile of total $k$ with $r_1+r_2+r_4\ge4$ if
and only if
\[
 h(\beta)\in B_k,\qquad B_5=B_6=B,\quad B_7=B'.
\]
In both assertions the resulting polynomial can have either prescribed
sign of the product of the component norms.
\end{proposition}
\begin{proof}
Let $G=(M/J)^\times$, let $E$ be the image of the full component unit
groups, and let $L$ be the image of the units of $A/J$. The explicit
finite calculation in Appendix~\ref{cy24:residue-data} gives
\begin{equation}\label{cy24:quotient}
 G/(EL)\simeq H,
\end{equation}
with component maps given by the table; thus $\ker h=EL$.
The full and signed unit quotients agree here: the integral order unit
$X$ has determinant $-1$, so multiplying a negative-norm component
unit by $X$ shows that $EL=E_+L$.

For the odd critical stratum, divide the occupied conductors by their
seeds. The resulting cofactor quotient omits exactly the three-primary
factors in components $1$ and $3$; its other factors are unchanged.
Both omitted factors have zero image under $h$. Consequently
\eqref{cy24:quotient} induces the same quotient for this smaller ring.
There is a compatible cofactor tuple $w$ with $w_1=w_2=1$,
$w_3=5+3\zeta_3$ and all other components global units; an explicit
polynomial is given in \eqref{cy24:odd-anchor}. Its class is
$h(w)=c$. At two, the ratio of two compatible seeded tuples is a unit
of the integral order. At three their retained nonzero branches have
the same paired residues; the omitted branch imposes no cofactor
condition. Thus all compatible cofactor residues form precisely the
coset $wL$ in the smaller quotient. Allowing $E$ gives $wLE$, proving
the first assertion.

For a fixed two-profile $r$, the $Q$ tuple in the proof of
Lemma~\ref{cy24:support} is a local unit. A local order unit with
scalar factor one normalizes it to one at two. Independently at three,
$S_r$ is a unit, so a local order unit normalizes $S_r\beta$ to one.
The remaining representatives have four arbitrary unit residues
\begin{equation}\label{cy24:1024}
 (\Z/8)^\times\times(\Z/8)^\times
       \times(\OO_4/4)^\times\times(\OO_8/2)^\times,
\end{equation}
and are fixed elsewhere. There are $1024$ such representatives.
Multiplication by $S_r$ followed by \eqref{cy24:T8} tests them exactly.
Table~\ref{cy24:profiles} records all nonempty profile images; their
unions are $B,B,B'$. Since each compatible set is $L$-stable and
$\ker h=EL$, two cofactor tuples of the same class differ by an
allowed global correction and a local order unit. This proves both
necessity and sufficiency. Nonrational units preserve their positive
norms; a final multiplication by $X$ chooses the determinant sign.
\end{proof}

Lemma~\ref{cy24:support} and Theorem~\ref{thm:prime-allocation}
identify the critical seeds and their cofactor classes $\mathcal R(m)$.
Intersection with the sets in Proposition~\ref{cy24:cofactor} gives
the four critical rows of Theorem~\ref{cy24:main}.
A prime word, Lemma~\ref{lem:unit-correction} and
\eqref{cy24:gluing} construct coefficients; multiplication by $X$
chooses the sign.

\subsection{The unrestricted strata}

We give the constructions to complete the proof. Let
$J=\sum_{j=0}^{23}X^j$ and, for a signed integer $u$, let
$G_u=(X^u-1)/(X-1)$ be an integral Laurent polynomial. All expressions
are reduced modulo $X^{24}-1$. If $u=24v+r$, $0\le r<24$, this
means $G_u=vJ+G_r$. When $(u,24)=1$, permutation of the nonidentity
roots gives $D_{24}(G_u)=u$.

The mixed stratum is exactly $288\Z$: necessity follows from
Lemma~\ref{cy24:support}, and the identities
\begin{equation}\label{cy24:mixed}
\begin{split}
 D_{24}\left(1+\sum_{j=2}^{12}X^j+vJ\right)&=288(2v+1),\\
 D_{24}(1-X+vJ)&=576v
\end{split}
\end{equation}
give sufficiency. The first is the $n=24$ specialization of
\cite[Lemma~4.2]{Kaiblinger2012}; the second follows by separating the
factor at one. This known mixed stratum is included to make the
classification exhaustive.

Every odd multiple of $27$ also occurs. Put $P=1+X^3+X^6$. Then
\begin{equation}\label{cy24:odd-tail}
 D_{24}(PG_u)=27u\quad((u,24)=1),\qquad
 D_{24}(G_{15}-X^{18}PG_{5-3v})=81v\quad(v\text{ odd}).
\end{equation}
The second expression is adapted from \cite[Lemma~4]{PaudelPinner15},
whose hypotheses concern distinct odd primes; the following calculation
establishes the required case with the composite factor eight.
Indeed, $D_{24}(P)=D_8(G_3)^3=27$. For the second polynomial $F_v$,
the value at one is $9v$, and, at every other $24$th root $z$,
\[
 z^9F_v(z)=-P(z)G_{8-3v}(z).
\]
Multiplying by $z-1$ proves this identity using
$(z^3-1)P(z)=z^9-1$. Since $(8-3v,24)=1$, its geometric factors
have product one. The $23$ minus signs cancel the product of
$z^{-9}$, leaving nonidentity product nine and determinant $81v$.
The two formulas cover exact $3$-adic valuation $3$ and all higher
$3$-adic valuations, respectively, including both signs.

Finally, every $2^km$ with $k\ge8$ and $(m,6)=1$ occurs. It suffices
to construct $2^k$ and multiply by a geometric polynomial. Define
\[
 B(X)=4+3X-3X^3,\qquad
 T(X)=1-X^4+X^8-X^{12}+X^{16}-X^{20}.
\]
For $r\ge0$ put
\begin{equation}\label{cy24:even-tail}
\begin{split}
 F_{8+2r}&=1+X^8+\frac{B(X)^r-1}{3}T(X),\\
 F_{9+2r}&=F_9+\frac{4^r-1}{3}J(-X),
\end{split}
\end{equation}
where the fixed $F_9$ is displayed in \eqref{cy24:even-anchor}.
Both quotients are integral. The polynomial $T$ is zero at every
cyclotomic component except $d=8$, where it is six. Moreover,
$\Norm_{K_8/\Q}(4+3\sqrt2)=4$. Thus the first formula replaces
the $d=8$ value $2$ of $1+X^8$ by $2(4+3\sqrt2)^r$, and its
determinant is $2^8\cdot4^r$. The second changes only $F_9(-1)=8$
to $8\cdot4^r$, and $D_{24}(F_9)=512$. This proves the even tail.
The zero polynomial, the exclusions in Lemma~\ref{cy24:support},
the critical criteria and \eqref{cy24:mixed}--\eqref{cy24:even-tail}
now prove every row of Theorem~\ref{cy24:main}.


\section{The group \texorpdfstring{$\C{2}\times\C{12}$}{C2 x C12}}
\label{c2c12:section}

A quotient of order $512$ controls the local compatibility for this
group. Its two-primary class sets become unrestricted at valuation $22$.
We define the quotient directly from finite rings, independently of
Smith coordinates.

\subsection{The signed conductor quotient}

Put $G=\langle A,B\mid A^2=B^{12}=1,\ AB=BA\rangle$ and
$\mathcal A=\Z[G]$.  Index the components by
\[
 j=(\sigma,d),\qquad \sigma\in\{0,1\},\quad
 d\in\{1,2,3,4,6,12\}.
\]
The component $j$ evaluates $A$ at $(-1)^\sigma$ and $B$ at a primitive
$d$th root $t_d$.  Write $\OO_d=\Z[t_d]$, including
$\OO_1=\OO_2=\Z$, and use the power basis in every component.
The maximal order is $\mathcal M=\prod_j\OO_d$.  For $v\in\mathcal M$,
put
\begin{equation}
 \nu(v)=\prod_{\sigma=0}^1
 v_{\sigma,1}v_{\sigma,2}
 \Norm_3(v_{\sigma,3})\Norm_4(v_{\sigma,4})
 \Norm_6(v_{\sigma,6})\Norm_{12}(v_{\sigma,12}).
 \label{c2c12:norm}
\end{equation}
Here $\Norm_d$ denotes the absolute field norm.  The nonrational factors
are positive for nonzero arguments, so the four rational entries retain
the sign.  Character factorization identifies the group determinant
with~\eqref{c2c12:norm}.

For tensor coordinates set
\begin{equation}
 u=A,\qquad x=B^9,\qquad z=B^4,\qquad B=xz.
 \label{c2c12:tensor}
\end{equation}
Thus $x^4=z^3=1$.  In the quartic component, with $t=t_{12}$, use
$i=t^9=-t^3$ and $\omega=t^4=t^2-1$.  In the component $d=6$, use
$t_6=-\omega$.  These conventions fix the identifications below.

The conductor $\mathfrak f$ of $\mathcal A$ in $\mathcal M$ has the
components in Table~\ref{c2c12:conductor-table}.  The two displayed
quotients are its relatively prime local factors.  At $3$ the upper
maps send $t_3$ to $1$, $t_6$ to $-1$, and $t_{12}$ to $i$.

\begin{table}[ht]
\centering
\caption{Conductor factors for each fixed value of $\sigma$.}
\label{c2c12:conductor-table}
\begin{tabular}{c c c c}
\toprule
$d$ & $\mathfrak f_d$ & quotient at $2$ & quotient at $3$\\
\midrule
$1,2$ & $24\Z$ & $\Z/8\Z$ & $\F_3$\\
$3$ & $8(1-t_3)\OO_3$ & $\OO_3/8\OO_3$ & $\F_3$\\
$6$ & $8(1+t_6)\OO_6$ & $\OO_6/8\OO_6$ & $\F_3$\\
$4$ & $12\OO_4$ & $\OO_4/4\OO_4$ & $\F_9$\\
$12$ & $4(1-t_{12}^4)\OO_{12}$ & $\OO_{12}/4\OO_{12}$ & $\F_9$\\
\bottomrule
\end{tabular}
\end{table}

Let $T=(\mathcal M/\mathfrak f)^\times$, and let $L$ be the unit group
of the image of $\mathcal A$ in $\mathcal M/\mathfrak f$.  Let $E_+$
be the residue image of
$\{\varepsilon\in\mathcal M^\times:\nu(\varepsilon)=1\}$.
Define the canonical quotient and its map by
\begin{equation}
 H=T/(E_+L),\qquad h:T\longrightarrow H.
 \label{c2c12:quotient}
\end{equation}
We use additive notation in $H$.  For a component element $a$ prime to
$6$, let $h_j(a)$ be the class of the tuple with entry $a$ at $j$ and
ones elsewhere.

\begin{lemma}
\label{c2c12:quotient-lemma}
The quotient in~\eqref{c2c12:quotient} is isomorphic to
$\C{2}^5\times\C{4}^2$.  A supplied tuple $v\in\mathcal M$ whose
component norms are prime to $6$ admits a norm-preserving unit
correction into $\mathcal A$ if and only if $h(v)=0$.
Every affirmative instance gives the coefficients of such an element.
\end{lemma}

\begin{proof}
Formula~\eqref{eq:abelian-conductor} gives the conductor table.
Appendix~\ref{c2c12:inverse-data} supplies complete local unit
generators and the finite presentation with invariant factors
$2,2,2,2,2,4,4$. Its global generators use the full unit groups of
Lemma~\ref{lem:fields} and three paired rational sign changes.
Lemma~\ref{lem:unit-correction} gives the criterion, and
\eqref{c2c12:full-inverse} recovers the coefficients.
\end{proof}

\subsection{Finite local class sets}

We first classify the raw valuations at $2$ and $3$ and their compatible
cofactor classes; the scalar test then allocates the remaining primes.
At $2$, decomposition by $z$ gives two $\C{2}\times\C{4}$ blocks,
over $\Z_2$ and $\Z_2[\omega]$. Let $R_a$ and $Q_r$ be the classes
of unit cofactors compatible with raw factors of norms $2^a$ and
$2^{2r}$ in the respective blocks, with their products normalized to one at
$3$. Compatibility means that multiplication by the
raw factors satisfies the local inverse congruences. Equations
\eqref{c2c12:fibres}--\eqref{c2c12:block-sets} in
Appendix~\ref{c2c12:two-data} give exact finite definitions using
clipped profiles and bounded exponent vectors.

At $3$, the conductor pairs a lower component with its ramified upper
component. Their raw factors are either both absent or both present.
Let $S_b$ be the compatible cofactor classes for these factors of norm
$3^b$, with their products normalized to one at $2$. The residue anchors and bounded supports in
\eqref{c2c12:three-anchor}--\eqref{c2c12:three-set} define this set
explicitly in Appendix~\ref{c2c12:three-data}. Put
\begin{equation}
 T_k=\bigcup_{a+2r=k}(R_a+Q_r),\qquad C_{k,b}=T_k+S_b.
 \label{c2c12:allowed}
\end{equation}
The letter $T_k$ denotes a subset of $H$, unlike the conductor unit
group $T$ in~\eqref{c2c12:quotient}.

\begin{lemma}
\label{c2c12:assembly}
For a supplied tuple $v$ of component norms prime to $6$, there are raw
uniformizer factors $r$ of norm $2^k3^b$ such that $rv$ admits a
norm-preserving unit correction into $\mathcal A$ if and only if
$h(v)\in C_{k,b}$. The sets satisfy
\begin{equation}
 T_0=\{0\},\qquad T_1=\cdots=T_7=\varnothing,
 \qquad T_k=H\quad(k\ge22),
 \label{c2c12:two-tail}
\end{equation}
and $S_b$ is nonempty exactly when $b\ne1$.
\end{lemma}

\begin{proof}
The local calculations in Appendices~\ref{c2c12:two-data} and
\ref{c2c12:three-data}, assembled by~\eqref{eq:crt-assembly}, give
the union of $h(Y_r)$ for raw norm $2^k3^b$.
Lemma~\ref{lem:unit-correction} proves the equivalence.
The complete minima modulo $4$ and boundary check $T_{22}=H$ in
Appendix~\ref{c2c12:two-data} give~\eqref{c2c12:two-tail}.
An occupied degree-one pair realizes every $b\ge2$; none has cost one.
\end{proof}

\subsection{The scalar classification}

Apply Theorem~\ref{thm:prime-allocation} with the quotient $H$,
component maps $h_j$, and distinguished position $j_0=(0,1)$.
Let $\ell_q(s)$ be the minimum cost of a prime word representing $s$,
with value $\infty$ if $s$ is unreachable.
For a signed integer $m$ prime to $6$, define
\begin{equation}
 \mathcal P(m)=h_{j_0}(m)+
 \sum_{q^e\parallel|m|}\{s\in H:\ell_q(s)\le\min(e,22)\}.
 \label{c2c12:prime-predicate}
\end{equation}
For $m=\pm1$, the prime sum is empty and
$\mathcal P(m)=\{h_{j_0}(m)\}$; the two signs may have different classes.

\Needspace{12\baselineskip}
\begin{theorem}
\label{c2c12:classification}
Zero belongs to $\Sdet{\C{2}\times\C{12}}$. Let
$D=2^k3^bm\ne0$, where $k,b\ge0$ and $\gcd(m,6)=1$, with $m$
retaining the sign of $D$. If $b=1$ or $1\le k\le7$, then $D$ does
not occur. For $k\ge22$, membership is equivalent to $b\ne1$.
For $k=0$ or $8\le k\le21$,
\begin{equation}
 D\in\Sdet{\C{2}\times\C{12}}
 \quad\Longleftrightarrow\quad
 \mathcal P(m)\cap C_{k,b}\ne\varnothing.
 \label{c2c12:criterion}
\end{equation}
Every admitted value has a constructive realization with $24$
integer coefficients.
\end{theorem}

\begin{proof}
Lemmas~\ref{lem:fields} and \ref{c2c12:quotient-lemma} verify the
hypotheses of Theorem~\ref{thm:prime-allocation}, with move costs at
most two. Since $H\simeq\C{2}^5\times\C{4}^2$,
\eqref{eq:two-group-budget} gives the prime budget $2(5+3+3)=22$.

After removing the factors above $2$ and $3$, the possible signed
cofactor classes are exactly $\mathcal P(m)$ by
Theorem~\ref{thm:prime-allocation}; the compatible classes are exactly
$C_{k,b}$ by Lemma~\ref{c2c12:assembly}. Their intersection proves
both directions. A witnessing prime word and
Lemma~\ref{lem:unit-correction} construct the coefficients, with $j_0$
carrying the sign. The zero vector realizes zero.

The empty local sets give the exclusions. For $k\ge22$ and $b\ne1$,
Lemma~\ref{c2c12:assembly} gives $C_{k,b}=H$; placing all of $m$ in
$j_0$ then avoids factoring $m$.
\end{proof}

The odd congruence $D\equiv1\pmod8$ follows also from the subgroup
inclusion of \cite[Theorem~1.4]{YamaguchiDedekind2024} and the
$\C{2}\times\C{4}$ classification of
\cite[Theorem~3.1]{PinnerSmyth2020}.  It is a useful preliminary test
at $k=0$, but does not replace~\eqref{c2c12:criterion}.

\paragraph{A signed example.}
Take $D=-2^{13}$, so that $m=-1$ and $b=0$.
The prime word is empty, leaving the baseline class $h_{j_0}(-1)$.
Put $t=t_{12}$ and $u_\pm=1-2t^2\pm2t^3$.
Since $u_-u_+=1$, both are units of norm one.
In the component order $d=1,2,3,4,6,12$, consider
\[
\begin{array}{c|rrrrrr}
 &1&2&3&4&6&12\\ \hline
 A=1  &-2&-2&1&2i&1&u_-\\
 A=-1 & 4&-8&1&4i&1&u_+.
\end{array}
\]
Use the supplied cofactor $v_{0,1}=-1$, all other entries one,
and the rational-block profile $(1,1,2,3,2,4)$, with no other raw
factors. The Gaussian raw entries are $(1-i)^2=-2i$ and $(1-i)^4=-4$.
The unit correction changes the rational signs at $(0,2)$ and $(1,2)$,
multiplies the Gaussian entries by $-1$ and $-i$, and puts $u_-$ and
$u_+$ in the quartic positions. Its total signed norm is one.
The row norms are $16$ and $-512$, and the inverse evaluation gives
\[
 F=B-B^2-B^4+B^5+B^9-A(B+B^3+B^5),
\]
with $D_G(F)=16(-512)=-8192$.
This realizes $h_{j_0}(-1)\in C_{13,0}$ while retaining the sign in
the distinguished rational cofactor.


\section{The group \texorpdfstring{$\C{2}\times\C{2}\times\C{6}$}{C2 x C2 x C6}}
\label{c2c2c6:section}

For this group, every odd determinant can be concentrated in one rational--Eisenstein pair.
The same reduction settles most even valuations.
Only $27$ valuation pairs require a finite quotient retaining the character positions.

Write $G=\C{2}^3\times\C{3}$, with generators $u,v,w,z$ of orders $2,2,2,3$.
This identifies $G$ with $\C{2}\times\C{2}\times\C{6}$ by taking the last generator to be $wz^2$.
Throughout this section, put
\[
 O=\Z[\omega],\qquad \omega^2+\omega+1=0,\qquad
 \lambda=1-\omega,\qquad \Norm(b+c\omega)=b^2-bc+c^2.
\]
Index the characters of $\C{2}^3$ by $h=(h_0,h_1,h_2)\in\F_2^3$, in the order $4h_0+2h_1+h_2$.
Let $W_{hg}=(-1)^{h\cdot g}$.
The rational character components of $F\in\Z[G]$ are pairs $(a_h,\beta_h)\in\Z\times O$, and
\begin{equation}\label{c2c2c6:product}
 D_G(F)=\prod_h a_h\Norm(\beta_h).
\end{equation}
All field norms in this formula are positive when the determinant is nonzero.

\begin{lemma}\label{c2c2c6:gluing}
A tuple $(a_h,\beta_h)$, with $\beta_h=b_h+c_h\omega$, comes from $\Z[G]$ precisely when
\begin{equation}\label{c2c2c6:gluing-equations}
 a_h\equiv\beta_h\pmod\lambda,\qquad
 Wa\in8\Z^8,\qquad W\beta\in8O^8.
\end{equation}
Its coefficient of $u^{g_0}v^{g_1}w^{g_2}z^j$ is the $g$-th Walsh transform, divided by $24$, of the $j$-th vector in
\begin{equation}\label{c2c2c6:inverse}
 (a+2b-c,\ a-b+2c,\ a-b-c).
\end{equation}
The component conductors are $24\Z$ and $8\lambda O$.
\end{lemma}

\begin{proof}
Combine the $C_3$ evaluation \eqref{eq:c3-evaluation}--\eqref{eq:c3-inverse}
with $W^{-1}=W/8$ to obtain the congruences and inverse formula.
The conductor formula \eqref{eq:abelian-conductor} gives $24\Z$ and
$8\lambda O$.
\end{proof}

\subsection{The normalized pair}

For a nonzero signed odd integer $t$, let $T(t)$ denote the condition
\begin{equation}\label{c2c2c6:T-pair}
 t=A \Norm(\beta),\qquad A\in\Z\text{ odd},\qquad
 \beta\equiv1\pmod{8O},\qquad A\equiv\beta\pmod\lambda.
\end{equation}
There is no congruence condition modulo $4$ or $8$ on $A$ in this definition.

\begin{proposition}[Normalization of odd determinants]\label{c2c2c6:normalized-pair}
For an odd integer $D$, membership in $\Sdet{G}$ is equivalent to
$D\equiv1\pmod8$ and $T(D)$.
Every such determinant has a realization
\begin{equation}\label{c2c2c6:one-pair-polynomial}
 1+(1+u)(1+v)(1+w)(r+sz+tz^2),\qquad r,s,t\in\Z.
\end{equation}
\end{proposition}

\begin{proof}
For a Walsh transform $g_h$ of eight elements of a commutative ring, put $q=g_0$.
Then
\begin{equation}\label{c2c2c6:walsh-product}
 \prod_hg_h\equiv q^8\pmod8.
\end{equation}
Indeed, write $g_h=q+2r_h$.
Each nonconstant coefficient occurs four times in $\sum_hr_h$, so this sum is divisible by $4$.
Modulo $2$, the function $h\mapsto r_h$ is linear.
The square and mixed coefficients in $\sum_{h<j}r_hr_j$ are respectively $\binom42=6$ and $4\cdot4-2=14$.
The product expansion modulo $8$ gives \eqref{c2c2c6:walsh-product}.

For an attaining odd tuple, set $A=\prod_ha_h$ and $B=\prod_h\beta_h$.
Equation~\eqref{c2c2c6:walsh-product} gives $A\equiv1\pmod8$.
Every unit modulo $8O$ has the form $\tau(1+2x)$, with $\tau\in\mu_3$, and $(1+2x)^4\equiv1\pmod{8O}$.
Thus $B$ is congruent modulo $8O$ to a unique $\eta\in\mu_3$.
Taking $\beta=\eta^{-1}B$ preserves its norm and its residue modulo $\lambda$.
Multiplication of the paired congruences in \eqref{c2c2c6:gluing-equations} proves \eqref{c2c2c6:T-pair}.
Also $\Norm(\beta)\equiv1\pmod8$, so $D\equiv1\pmod8$.

Conversely, \eqref{c2c2c6:T-pair} and $D\equiv1\pmod8$ imply
$A\equiv1\pmod8$. The construction \eqref{eq:normalized-pair-construction}
with $P=\C{2}^3$ gives \eqref{c2c2c6:one-pair-polynomial}.
\end{proof}

The condition $T$ is a small arithmetic test.
Define the finite groups
\[
 Q_{16}=(O/8\lambda O)^\times/\mu_6,\qquad
 Q_8=(O/8O)^\times/\mu_6.
\]
Appendix~\ref{c2c2c6:quotient-data} gives generators proving
$Q_{16}\simeq\C{4}\times\C{2}^2$ and $Q_8\simeq\C{4}\times\C{2}$.
Put $\kappa=[17]\in Q_{16}$ and $\ell=[\lambda]\in Q_8$.
For either group $Q$, a prime ideal $\mathfrak q\mid q$, with $q\ne2,3$, contributes the class of a generator $\pi_{\mathfrak q}$ at cost $f_{\mathfrak q}$, where $\Norm(\pi_{\mathfrak q})=q^{f_{\mathfrak q}}$.
All prime ideals above $q$ are included.
Let $\mathcal R_Q(n)$ be the sum of the resulting prime-state sets with costs at $q$ bounded by $v_q(n)$, and set $\mathcal R_Q(1)=\{0\}$.
By \eqref{eq:two-group-budget}, the budgets may be capped at $10$ for
$Q_{16}$ and $8$ for $Q_8$.

\begin{proposition}\label{c2c2c6:T-arithmetic}
Write $t=3^bm$, with $t$ signed and odd and $\gcd(m,6)=1$.
Then $T(t)$ is given by
\begin{equation}\label{c2c2c6:T-table}
\begin{array}{c@{\quad}l}
 b&\text{condition}\\ \midrule
 0&\chi_3(t)\kappa\in\mathcal R_{Q_{16}}(|t|),\\
 1&\text{false},\\
 2,3,4&\mathcal R_{Q_8}(|m|)\cap\{-y\ell:1\le y<b\}\ne\varnothing,\\
 b\ge5&\text{true}.
\end{array}
\end{equation}
Here $\chi_3(t)=0$ when $t\equiv1\pmod3$, and $\chi_3(t)=1$ otherwise.
Every affirmative row constructs a pair in \eqref{c2c2c6:T-pair}.
\end{proposition}

\begin{proof}
The ring $O$ has class number one and unit group $\mu_6$ by Lemma~\ref{lem:fields}.
Unique ideal factorization therefore identifies $\mathcal R_Q(|m|)$ with the classes of elements whose positive norms divide $|m|$.
If $b=0$, then $\Norm(\beta)\equiv1\pmod3$, so $A\equiv t\pmod3$.
The required residue of $\beta$ is one modulo $8O$ and $t$ modulo $\lambda$, which gives the first row.
Conversely, a representing prime word can be multiplied by a unit to have exactly those residues; then $A=t/\Norm(\beta)$ has the required sign and congruence.

If $b>0$, the paired congruence requires both $3\mid A$ and $\lambda\mid\beta$.
Writing $\beta=\lambda^y\gamma$ gives $1\le y<b$, $\Norm(\gamma)\mid|m|$, and $[\gamma]=-y\ell$ in $Q_8$.
These conditions are also sufficient after multiplication by a unit.
Finally, $\lambda^4=9\omega^2$ and $9\equiv1\pmod8$.
For $b\ge5$, the choice $\beta=9$ and $A=t/81$ proves the last row.
\end{proof}

\subsection{The finite test for the remaining even values}

The finite quotient must retain all eight positions.
Set
\begin{equation}\label{c2c2c6:Gamma}
 \Gamma=\prod_{h\in\F_2^3}
 \bigl((\Z/24\Z)^\times\times(O/8\lambda O)^\times\bigr).
\end{equation}
Let $L$ be the subgroup whose residues satisfy \eqref{c2c2c6:gluing-equations}.
Let $E^+$ consist of the component units
\[
 ((\epsilon_h,\xi_h))_h,\qquad
 \epsilon_h\in\{1,-1\},\quad \prod_h\epsilon_h=1,\quad \xi_h\in\mu_6.
\]
Define $H=\Gamma/(E^+L)$, write its law additively, and denote the quotient map by $\theta$.
For a component residue $v$, write $\theta_{h,1}(v)$ or $\theta_{h,3}(v)$ for its class with every other component one.
Its invariant factors are $\C{2}^{10}\times\C{4}$, as computed in
Appendix~\ref{c2c2c6:quotient-data}; in particular, it has order $4096$.

We first classify the raw valuations at $2$ and $3$ and their compatible
cofactor classes, then allocate the remaining primes.
For a positive ordered vector $e$, let $R_e$ consist of the classes of tuples satisfying
\begin{equation}\label{c2c2c6:R-definition}
 u_h\in\{1,7,13,19\},\qquad
 \gamma_h=17^{e_h\bmod2},\qquad
 W(2^{e_h}u_h)_h\in8\Z^8.
\end{equation}
For a positive ordered vector $f$, let $Q_f$ consist of the classes of tuples satisfying
\begin{equation}\label{c2c2c6:Q-definition}
 u_h=1,\qquad \gamma_h\equiv2^{f_h}\pmod\lambda,
 \qquad W(2^{f_h}\gamma_h)_h\in8O^8,
\end{equation}
where each $\gamma_h$ runs over $(O/8\lambda O)^\times$.
Put $R_0=Q_0=\{0\}$, and let $R_a$ and $Q_r$ be the unions of $R_e$ and $Q_f$ over vectors of totals $a$ and $r$.
Only $a\le31$ and $r\le15$ will occur.
Appendix~\ref{c2c2c6:profile-rules} gives smaller finite generation rules for these sets and retains explicit representatives.

The factors at three are independent of these profiles.
Put
\[
 d_h=\theta_{h,1}(19),\qquad \kappa_h=\theta_{h,3}(17),\qquad
 \eta_h(y)=\theta_{h,3}(\lambda^{-y}\bmod8O,\ 1\bmod\lambda).
\]
The final argument denotes a Chinese remainder residue.
For arrays $x,y$ with $(x_h,y_h)=(0,0)$ or $x_h,y_h>0$, and binary choices $z_h$ at their occupied positions, take the shift
\begin{equation}\label{c2c2c6:support-shift}
 \sum_h\bigl((x_h\bmod2)d_h+\eta_h(y_h)+z_h\kappa_h\bigr).
\end{equation}
At an unoccupied position, all three summands are zero.
Let $S_b$ consist of these shifts with $\sum_h(x_h+y_h)=b$; thus $S_0=\{0\}$ and $S_1=\varnothing$.
Appendix~\ref{c2c2c6:support-rule} gives an exact finite recurrence
and its padding rule.
Define
\begin{equation}\label{c2c2c6:C-kb}
 C_{k,b}=\bigcup_{a+2r=k}(R_a+Q_r+S_b),
 \quad
 a\in\{0,8\}\cup\Z_{\ge12},\quad
 r\in\{0,8\}\cup\Z_{\ge10}.
\end{equation}

Apply Theorem~\ref{thm:prime-allocation} with component maps
$\theta_{h,1},\theta_{h,3}$ and distinguished position $(0,1)$.
For signed $m$ prime to six, put $c(m)=\theta_{0,1}(m)$.
The prime moves are
\begin{equation}\label{c2c2c6:prime-moves}
\begin{aligned}
 &\theta_{h,1}(q)-\theta_{0,1}(q), &&h\ne0,\quad\text{cost }1,\\
 &\theta_{h,3}(\pi_{\mathfrak q})-f_{\mathfrak q}\theta_{0,1}(q),
 &&h\in\F_2^3,\ \mathfrak q\mid q,\quad\text{cost }f_{\mathfrak q}.
\end{aligned}
\end{equation}
Let $P_q(e)$ be the states reachable with cost at most $e$.
Every move costs at most two, so \eqref{eq:two-group-budget} gives
$P_q(e)=P_q(\min(e,26))$, since $2(10+3)=26$.
The required finite condition is
\begin{equation}\label{c2c2c6:finite-test}
 \left(c(m)+\sum_{q^e\parallel|m|}P_q(\min(e,26))\right)
 \cap C_{k,b}\ne\varnothing.
\end{equation}

\Needspace{22\baselineskip}
\begin{theorem}\label{c2c2c6:classification}
Zero belongs to $\Sdet{G}$.
For $D\ne0$, write $D=2^k3^bm=2^kn$, with $k,b\ge0$, $\gcd(m,6)=1$, and both $m$ and $n$ retaining the sign of $D$.
If $b=1$, or if $k\in\{1,2,3,4,5,6,7,9,10,11\}$, then $D\notin\Sdet{G}$.
Otherwise:
\begin{equation}\label{c2c2c6:classification-table}
\begin{array}{c@{\quad}l}
 \text{range}&D\in\Sdet{G}\text{ precisely when}\\ \midrule
 k=0&D\equiv1\pmod8\text{ and }T(D),\\
 k=8&n\equiv1\pmod4\text{ and }T(n),\\
 k=12&T(n),\\
 k\in\{14,15,17,18,19,21,23,25,27\}&T((-1)^kn),\\
 k\ge32&\text{always},\\
 k=24,\ b\ge3&\text{always},\\
 k\ge28,\ b\ge2&\text{always},\\
 k\ge14,\ b\ge5&\text{always},\\
 k=13,\ b\ge6&\text{always}.
\end{array}
\end{equation}
The remaining cases are the $27$ pairs
\begin{equation}\label{c2c2c6:exceptional-pairs}
 \begin{gathered}
 \{13\}\times\{0,2,3,4,5\}\ \cup\
 \{16,20,22,26\}\times\{0,2,3,4\}\\
 {}\cup\ \{24\}\times\{0,2\}\ \cup\
 \{28,29,30,31\}\times\{0\}.
 \end{gathered}
\end{equation}
For each of them, membership is exactly \eqref{c2c2c6:finite-test}.
Every affirmative condition constructs an element of $\Z[G]$ with determinant $D$.
\end{theorem}

\begin{proof}
The paired congruence at three forces $3\mid a_h$ and $\lambda\mid\beta_h$ to occur together.
Each occupied pair contributes at least two to $b$, excluding $b=1$.
Lemma~\ref{c2c2c6:local-valuations} gives rational and quadratic totals
\[
 a\in\{0,8\}\cup\Z_{\ge12},\qquad
 r\in\{0,8\}\cup\Z_{\ge10},\qquad k=a+2r.
\]
It proves the excluded $2$-adic valuations and shows that the quadratic total is zero at $k=8,12$ and at the nine values in the fourth row.

When $r=0$, normalize $\prod_h\beta_h$ as in Proposition~\ref{c2c2c6:normalized-pair}.
If $\prod_ha_h=2^kA_0$, multiplication of the paired congruences gives
$(-1)^kA_0\equiv\beta\pmod\lambda$.
Thus $T((-1)^kn)$ is necessary.
At $k=8$, the rational local product additionally requires $A_0\equiv1\pmod4$, equivalently $n\equiv1\pmod4$.
At $k=12$ there is no extra restriction.
For sufficiency, a pair $(A,\beta)$ from the indicated $T$ test gives the following component tuples:
\begin{equation}\label{c2c2c6:rational-templates}
\begin{array}{c@{\quad}l@{\quad}l}
 k&a&\beta\text{-tuple}\\ \midrule
 8&(-2A,-2,-2,-2,-2,-2,-2,-2)&(\beta,1,1,1,1,1,1,1),\\
 12&(4A,4,4,4,2,2,2,2)&(\beta,1,1,1,-1,-1,-1,-1),\\
 k\ge14&((-1)^{k+1}2^{k-11}A,-8,4,4,2,2,2,2)
 &(\beta,1,1,1,-1,-1,-1,-1).
\end{array}
\end{equation}
Their Walsh sums are divisible by eight, and the paired residues match at every position.
Their determinants are respectively $2^8A\Norm(\beta)$, $2^{12}A\Norm(\beta)$, and $(-1)^k2^kA\Norm(\beta)$.
The last row also proves the stated high-three family by taking $\beta=9$.
Appendix~\ref{c2c2c6:tails} proves the high-two and $k=13$ families.
The full-target identities \eqref{c2c2c6:full-target-identities} give the rows at $k=24$ and $28\le k\le31$; the preceding families extend them to the stated ranges.

For the pairs \eqref{c2c2c6:exceptional-pairs}, Proposition~\ref{c2c2c6:assembly} proves that $C_{k,b}$ is exactly the set of cofactor classes admitting the required powers of $2$ and $3$.
Theorem~\ref{thm:prime-allocation}, applied to \eqref{c2c2c6:prime-moves}, gives exactly the classes of all tuples with signed cofactor norm $m$.
Their intersection proves necessity and sufficiency of \eqref{c2c2c6:finite-test}.
A witnessing prime word and profile anchor, followed by
Lemma~\ref{lem:unit-correction} and \eqref{c2c2c6:inverse}, construct
the $24$ integer coefficients with determinant $D$.
\end{proof}

For $B\subseteq H$, its translation stabilizer
$P_B=\{g\in H:B+g=B\}$ is a subgroup, and $B$ is a union of its
cosets. With $\pi_B\colon H\to H/P_B$, this gives, for every $A\subseteq H$,
\begin{equation}\label{c2c2c6:target-stabilizers}
 A\cap B\ne\varnothing\quad\Longleftrightarrow\quad
 \pi_B(A)\cap\pi_B(B)\ne\varnothing.
\end{equation}
Thus \eqref{c2c2c6:finite-test} can be tested in $H/P_{C_{k,b}}$.
The $27$ pairs give $18$ distinct targets. Their exact stabilizers are
provided in the supplement; for example, $|H/P_B|=16$ at $(k,b)=(24,2)$,
whereas $P_B=\{0\}$ at $(16,0)$ and $(22,0)$.
Retaining the original generators along a projected prime word gives
its full $H$-class in $B$, hence the same constructive lift.

\begin{corollary}[Four one-move tests]\label{c2c2c6:one-move}
For $k=28,30$ and $b=0$, the quotient $H/P_{C_{k,0}}$ is $\C{2}^6$;
for $k=29,31$, it is $\C{2}^5$.
In each case the projected target omits exactly one class $t$.
Thus \eqref{c2c2c6:finite-test} holds if and only if the projected baseline differs
from $t$, or if some move in \eqref{c2c2c6:prime-moves} has nonzero
projection and cost at most $v_q(|m|)$.
\end{corollary}

\begin{proof}
The full target calculation gives complements of sizes $64,128,64,128$,
each a single stabilizer coset. From the omitted class, one available
nonzero move suffices; if all available moves project to zero, no word
leaves it. Every move costs at most two.
\end{proof}


\section{Comparison of the three determinant sets}\label{sec:comparison}

The pure powers of two give a short comparison that retains the sign
restrictions of the two noncyclic groups.

\Needspace{15\baselineskip}
\begin{corollary}\label{cor:pure-two}
Among nonzero integers whose absolute value is a power of two,
the three determinant sets are as follows:
\begin{align*}
 S(\C{24})&:\quad \{-1,1\}\ \cup\ \{\pm2^j:j\ge8\},\\
 S(\C{2}\times\C{12})&:\quad
 \{1,2^8,2^{12}\}\ \cup\ \{(-2)^j:13\le j\le19\}
 \ \cup\ \{\pm2^j:j\ge20\},\\
 S(\C{2}\times\C{2}\times\C{6})&:\quad
 \{1,2^8,2^{12}\}\ \cup\ \{(-2)^j:14\le j\le31\}
 \ \cup\ \{\pm2^j:j\ge32\}.
\end{align*}
\end{corollary}

\begin{proof}
For $\C{24}$, take $b=0$ and $m=1$ in Theorem~\ref{cy24:main}.
Then $\mathcal R(1)=\{0\}$ meets none of the critical target sets.
For the noncyclic groups, cofactors $1$ and $-1$ admit no prime moves,
so membership reduces to testing the two signed baseline classes in
the local sets. Appendix~\ref{c2c12:two-data} gives these intersections
through exponent $21$, and Appendix~\ref{c2c2c6:pure-two-data} gives
the second group's finite intersections. The valuation exclusions,
odd congruences, and high-valuation constructions complete both lists.
\end{proof}

Thus $-2^{13}$ distinguishes the two noncyclic groups, and
$-2^{20}$ distinguishes their sign thresholds. The odd sets admit
a stronger comparison.

Pinner and Smyth~\cite[Theorem~3.1]{PinnerSmyth2020} show that
$S(\C{2}^3)\subsetneq S(\C{2}\times\C{4})$, while both odd subsets
equal $8\Z+1$. Theorem~\ref{thm:odd-inclusion} shows that adjoining
$\C{3}$ separates these odd subsets; its proof transfers normalized
pairs and does not follow from the subgroup relation
in~\cite[Theorem~1.4]{YamaguchiDedekind2024}.

\begin{theorem}\label{thm:odd-inclusion}
Put $E=\C{2}\times\C{2}\times\C{6}$ and
$G=\C{2}\times\C{12}$. Then
\[
 S_{\mathrm{odd}}(E)\subsetneq S_{\mathrm{odd}}(G).
\]
For every integer $r\ge0$,
\[
 D_r=3^4\cdot5\cdot13\cdot73^{r+1}
 \in S_{\mathrm{odd}}(G)\setminus S_{\mathrm{odd}}(E).
\]
\end{theorem}

\begin{proof}
Write $O=\Z[\omega]$ and $\lambda=1-\omega$ as before.
By Proposition~\ref{c2c2c6:normalized-pair}, an odd determinant of $E$
has the form $A\Norm(\beta)$, where
\[
 A\equiv1\pmod8,\qquad \beta\equiv1\pmod{8O},\qquad
 A\equiv\beta\pmod\lambda.
\]
The construction \eqref{eq:normalized-pair-construction} realizes
this pair in $\Z[P\times\C{3}]$ for $P=\C{2}\times\C{4}$,
proving the inclusion.

For strictness, let $G=\langle A,B:A^2=B^{12}=1,\ AB=BA\rangle$
and put
\[
 F=1+AB^6-B^6+B^{11}+B^3.
\]
Evaluation at $A=1,-1$ gives
$1+B^{11}+B^3$ and $1-2B^6+B^{11}+B^3$.
Their cyclotomic component norms are
\[
\begin{array}{c|rrrrrr|r}
 d&1&2&3&4&6&12&\text{product}\\ \hline
 A=1 &3&-1&3&5&1&1&-45\\
 A=-1&1&-3&1&13&3&73&-8541 .
\end{array}
\]
These values follow by reduction modulo $\Phi_d(B)$.
Consequently $D_G(F)=384345=3^4\cdot5\cdot13\cdot73$.
Since $73\equiv1\pmod8$, the polynomial
\[
 F_r=F+\frac{73^r-1}{8}J_G
\]
has integer coefficients. Adding the multiple of $J_G$ changes only
the trivial character value, from $3$ to $3\cdot73^r$.
Thus $D_G(F_r)=D_r$ for every $r\ge0$.

It remains to exclude every possible normalized pair for $D_r$.
Suppose $D_r=A\Norm(\beta)$ with $\beta\equiv1\pmod{8O}$ and
$A\equiv\beta\pmod\lambda$.
The $3$-adic valuation forces $1\le y=v_\lambda(\beta)\le3$.
The prime five is inert in $O$ and occurs only once in $D_r$, so it
does not divide $\Norm(\beta)$. The other primes split as
\[
 13=(4+\omega)(3-\omega),\qquad
 73=(9+\omega)(8-\omega).
\]
Unique factorization gives
\[
 \beta=\varepsilon\lambda^y\pi(9+\omega)^j(8-\omega)^k,
 \qquad
 \pi\in\{1,4+\omega,3-\omega\},\quad j+k\le r+1,
\]
with $\varepsilon\in\mu_6$. Its norm modulo eight is either $3^y$
or $5\cdot3^y$. The required residue one forces $y=2$ and $\pi=1$.
But $9+\omega$ and $8-\omega$ reduce modulo $8O$ to the global
units $1+\omega$ and $-\omega$.
Since $\lambda^2=-3\omega$, it follows that
$\beta\bmod8O$ lies in $3\mu_6$.
None of these six residues is one, a contradiction.
\end{proof}

Normalization permits the transfer between the two groups, while the
family $D_r$ shows that the transfer is not reversible.


\section{Exact computations and reproducibility}\label{sec:computation}

The proofs use full-unit certificates (Lemma~\ref{lem:fields}), the
rational coset formula, and exhaustive finite-ring calculations
(Appendices~A--C).
The recorded environment is SageMath~10.9, Python~3.12.13 and
PARI/GP~2.17.3~\cite{Sage,PARI}; full-unit certification uses unconditional
\texttt{bnfcertify} with flag zero. Class number one follows from
Masley's theorem, independently of this computation. The supplementary \texttt{classify.py}
returns a decision and, for each admitted integer, $24$ coefficients.
The file \texttt{proof\_computation\_index.md} links each mathematical
output to its generator, fresh command and certificate; the dispatcher
\texttt{rebuild\_proof\_data.py all} runs all proof-data stages.
Exact certificates and source hashes accompany the programs.

Separate regression tests compare reconstructed coefficients with
determinants of the original $24\times24$ integer matrices, including
both signs, valuation boundaries and representative family parameters.
These detect implementation and transcription errors. Completeness
follows from the unit-correction arguments, the rational coset formula,
the quadratic and support calculations with their proved padding rules, and
Theorem~\ref{thm:prime-allocation}.

Table~\ref{tab:implementation-bounds} compares the prime budgets.
A larger retained cap gives the same reachable set by
\eqref{eq:two-group-budget}. For the retained split/inert caps, the
supplement proves stabilization by exhausting the conductor residue types and checking
closure under every move. Thus the comparison covers every exponent.
The public third-group test uses the target quotients
in~\eqref{c2c2c6:target-stabilizers}; its four one-move cases preserve the
membership decision, without requiring the whole reachable set.

\begin{table}[tb]
\centering
\caption{Prime budgets in the text and code. The notation $4/2$ means
split/inert in $\Q(\omega)$. Program and certificate paths are given in
\texttt{implementation\_correspondence.md}.}
\label{tab:implementation-bounds}
\begin{tabular}{@{}lccp{4.1cm}@{}}
\toprule
Predicate & Text & Code & Equivalence\\
\midrule
\eqref{cy24:prime-test} & $10$ & $30$ & Zero-sum bound
\eqref{eq:two-group-budget}\\
\eqref{c2c12:prime-predicate} & $22$ & $1022$ & Zero-sum bound
\eqref{eq:two-group-budget}\\
\eqref{c2c2c6:finite-test}, public & $26$ & $26$ &
Exact target preimages\\
Four one-move cases & $26$ & $2$ &
Corollary~\ref{c2c2c6:one-move}\\
\eqref{c2c2c6:finite-test}, retained & $26$ & $4/2$ &
Complete residue-type graphs\\
\bottomrule
\end{tabular}
\end{table}

The correspondence file also covers the smaller normalized-pair tests.


\section*{Statements and declarations}

\paragraph{Data and code availability.}
Online Resource 1 (version 1.0.0) is available at
\url{https://github.com/chatchawanpan-dev/order24-group-determinants/releases/tag/v1.0.0}.
It contains the exact arithmetic programs, generated certificates,
coefficient constructors and reproducibility instructions.

\paragraph{Funding.}
The author received no funding for this work.

\paragraph{Competing interests.}
The author declares no competing interests.

\paragraph{Use of generative AI.}
OpenAI's Codex assistant was used to develop the exposition, assist
with mathematical arguments, write and review computational code,
and prepare the manuscript. Its suggestions were checked through
the stated proof reviews and exact arithmetic computations.


\appendix

\section{Finite residue data for the cyclic case}\label{cy24:residue-data}

\subsection{Component rings and homomorphisms}

Write $R_{d,p}$ for the $p$-primary factor of $\OO_d/J_d$.
In Table~\ref{cy24:rings}, $z$ is the primitive root in the displayed
ring. In particular, at three the map from $\OO_{3e}$ to
$\OO_e/3$ sends $\zeta_{3e}$ to $z=\xi_e$, as specified in the
integral gluing. The symbols $\OO_1$ and $\OO_2$ both mean $\Z$.
\begin{table}[htbp]
\centering
\caption{The full conductor residue rings. The last column is
$|(R_{d,2}\times R_{d,3})^\times|$.}\label{cy24:rings}
\begin{tabular}{cllr}
\toprule
$d$ & $R_{d,2}$ & $R_{d,3}$ & Unit order\\
\midrule
1 & $\Z/8$ & $\F_3$ & 8\\
2 & $\Z/8$ & $\F_3$ & 8\\
3 & $\OO_3/8$ & $\F_3$ & 96\\
4 & $\OO_4/4$ & $\OO_4/3$ & 64\\
6 & $\OO_6/8$ & $\F_3$ & 96\\
8 & $\OO_8/2$ & $\OO_8/3$ & 512\\
12 & $\OO_{12}/4$ & $\OO_4/3$ & 1536\\
24 & $\OO_{24}/2$ & $\OO_8/3$ & 12288\\
\bottomrule
\end{tabular}
\end{table}

The product $G=\prod_{d,p}R_{d,p}^\times$ has order $2^{52}3^4$.
For rational two-primary factors use generators $3,5$; for every
$\F_3$ factor use generator $2$. Their images in $H$ are zero.
All remaining generators and their images are given in
Table~\ref{cy24:map}. Entries in each row are ordered alike.
For compactness, $abc$ denotes the triple $(a,b,c)$, with the last
coordinate read modulo four. The component map $\psi_d$ is the sum
of its two local maps after reducing an element of $\OO_d$.

\begin{table}[htbp]
\centering\small
\caption{Complete local generators and their images in $H$.
The presentation types are defined in \eqref{cy24:presentations}.}
\label{cy24:map}
\begin{tabular}{cccll}
\toprule
$d$ & $p$ & Type & Ordered generators & Ordered images\\
\midrule
3&2&A&$1+3z,\ 1+7z,\ 2+z$&$112,\ 012,\ 012$\\
4&2&B&$z,\ 2+z$&$002,\ 002$\\
4&3&C&$1+z$&$103$\\
6&2&A&$1+z,\ 1+5z,\ 2+3z$&$010,\ 110,\ 010$\\
8&2&B&$z^3,\ 1+z^2+z^3$&$002,\ 002$\\
8&3&D&$z^3,\ z^2+z^3$&$002,\ 111$\\
12&2&E&$z^2+z^3,\ z^2+3z^3,\ z$&$103,\ 101,\ 002$\\
&&&$z+z^3,\ z+2z^2$&$000,\ 100$\\
12&3&C&$1+z$&$101$\\
24&2&F&$z^7,\ z^6+z^7$&$002,\ 111$\\
&&&$z^4+z^5+z^7,\ z^3+z^6+z^7$&$013,\ 111$\\
24&3&D&$z^3,\ z^2+z^3$&$002,\ 113$\\
\bottomrule
\end{tabular}
\end{table}

Here are complete multiplicative presentations. All generators commute;
$g_1,g_2,\ldots$ occur in their table order.
\begin{equation}\label{cy24:presentations}
\begin{array}{c|l}
A&g_1^{12}=1,\quad g_2^2=g_1^2,\quad g_3^2=g_1^{10}\\[2pt]
B&g_1^4=1,\quad g_2^2=g_1^2\\[2pt]
C&g_1^8=1\\[2pt]
D&g_1^8=g_2^8=1\\[2pt]
E&g_1^{12}=1,\quad g_2^2=g_1^2,\quad
  g_3^2=g_1g_2,\quad g_4^2=g_1^9g_2,\quad g_5^2=g_1g_2\\[2pt]
F&g_1^{12}=1,\quad g_2^4=g_1^8,\quad
  g_3^2=g_1^6g_2^2,\quad g_4^2=g_1^{10}g_2^2.
\end{array}
\end{equation}
The rational two-primary factors have $g_1^2=g_2^2=1$, and the
$\F_3$ factors have $g_1^2=1$. Polynomial multiplication verifies
the relations. Reducing exponents from last to first gives normal
forms in the indicated ranges, for example $12\cdot2^4=192$ forms
for type $E$. Their polynomial residues are distinct and exhaust the
unit counts. The displayed images satisfy the relations, so each map
is well defined and is evaluated by adding the images of a residue's
unique normal form.

\subsection{The quotient kernel}

Let $E\subset G$ be the image of $\prod_d\OO_d^\times$; the full
unit generators are those of Lemma~\ref{lem:fields}. We give explicit
generators for the second subgroup $L$, the image of $(A/J)^\times$.
These are units of a finite integral order; their polynomial lifts
need not have determinant $\pm1$ over $\Z$.

At two put $x=X^{16}$, $y=X^9$ and $\delta=y-1$. In the product
of two-primary rings, $8=0$ and $\delta^4=0$. For
\begin{equation}\label{cy24:theta}
 \theta\in\{1+x+x^2,\ 2-x-x^2,\ x(2-x-x^2)\},
 \quad 0\le r<3,\quad0\le s<4,\quad r+s>0,
\end{equation}
take the unit $1+\theta2^r\delta^s$, together with $x$.
Identity images may be omitted. The resulting subgroup is denoted $L_2$.
Every one of these polynomials is the identity at the three-primary
factors, since $x=1$ and every $\theta=0$ there.

For the two-primary order image $A_2$, the ideal $I=(2,\delta)$
satisfies $I^6=0$ and $A_2/I=\F_2\times\F_4$, whose units are
generated by $x$. Since three is invertible modulo eight, the three
$\theta$ span the scalar and quadratic summands of
$\Z/8[x]/(x^3-1)$; hence $\theta2^r\delta^s$ span the successive
$I$-layers. Formula~\eqref{eq:unit-filtration} proves $L_2=A_2^\times$.
The maximal ring has order $2^{42}$ and the order index is $2^{21}$,
so $|A_2|=2^{21}$ and $|L_2|=3\cdot2^{18}$.

At three the order image is $\F_3[y]/(y^8-1)$. Use its factorization
\begin{equation}\label{cy24:three-factors}
 y^8-1=(y-1)(y+1)(y^2+1)(y^2+y-1)(y^2-y-1)
       \quad\text{in }\F_3[y].
\end{equation}
In these factors choose the primitive units $2,2,1+y,y,y$,
respectively. For each factor, polynomial CRT gives a unique polynomial
$H_j$ of degree below eight which is this unit in that factor and one
in the other four. Lift its coefficients to $\{0,1,2\}$ and use
\begin{equation}\label{cy24:three-lifts}
 1+16\bigl(H_j(X^9)-1\bigr),\qquad 1\le j\le5.
\end{equation}
These polynomials are the identity at two and generate the full
three-primary order-unit subgroup $L_3$, of order $2^2\cdot8^3=2^{11}$.
Consequently $L=L_2L_3$ and $|L|=3\cdot2^{29}$.

\begin{lemma}\label{cy24:finite-kernel}
The sum $h$ of the component maps in Table~\ref{cy24:map} is
surjective and has kernel $EL$.
\end{lemma}
\begin{proof}
Substitution of the full field units and the polynomials
\eqref{cy24:theta}--\eqref{cy24:three-lifts} in the local normal
forms gives zero under $h$. Surjectivity follows, for example, from
the images $100$, $010$ and $103$ in the table.

For the index calculation, order the local factors by increasing $d$, with two preceding three,
and use the generator order in Table~\ref{cy24:map}, inserting the
rational and $\F_3$ generators specified above. There are $33$
ambient generators. Form an integer matrix with the following columns:
the exponent vectors of the relations \eqref{cy24:presentations};
the normal-form vectors of the thirteen full field-unit generators;
the normal-form vectors of $x$ and all $33$ units in
\eqref{cy24:theta}; and the five normal-form vectors in
\eqref{cy24:three-lifts}. The matrix is thus completely specified by
polynomial multiplication in the displayed rings, including every
column whose image is the identity. Integer elimination gives its
nonzero Smith diagonal
\[
 \underbrace{1,\ldots,1}_{30},\ 2,\ 2,\ 4.
\]
Equivalently, the gcd of its maximal minors is $16$.
One can reproduce this finite calculation by reducing each column to
the stated normal-form ranges; no algebraic norm search is involved.
It proves $[G:EL]=16$. The already proved inclusion $EL\subseteq\ker h$
and surjectivity to the sixteen-element group give equality.
\end{proof}

This calculation also explains why one map suffices for both primes.
For the odd three-critical cofactor ring, delete the three-primary
factors in components $1$ and $3$. The associated order-unit image
is the projection of $L$: the $y-1$ generator in
\eqref{cy24:three-lifts} disappears. The kernel of this projection
is $(\F_3^\times)^2$, and both its factors have zero image under
$h$. By Lemma~\ref{cy24:finite-kernel} it is already contained in $EL$.
The induced obstruction quotient is therefore again $H$, with the
same component maps. This identifies the two tests canonically through
the projection of their conductor rings.

\subsection{Compatible representatives and support tables}

The following polynomial supplies the odd critical representative:
\begin{equation}\label{cy24:odd-anchor}
\begin{split}
 W(X)={}&2+X^2+2X^3+X^4+X^9-X^{10}-2X^{11}\\
        &-X^{13}-X^{14}-X^{15}-X^{16}+X^{18}-X^{20}
          +X^{21}+X^{22}+X^{23}.
\end{split}
\end{equation}
Its rational values are $3$ and $1$, and
$W(\zeta_3)=(1-\zeta_3)(5+3\zeta_3)$. All its other field
components have norm one. These identities follow by reduction modulo
the eight cyclotomic polynomials. Thus $D_{24}(W)=171$, and division
by $S_3$ gives a compatible cofactor class
$\psi_3(5+3\zeta_3)=012=c$.

For each two-profile $r$ in Lemma~\ref{cy24:support}, form the following
finite set of representatives. Choose a tuple $u$ from
\eqref{cy24:1024}. At two, put these values in components $1,2,4,8$
and put one in components $3,6,12,24$. At three put $S_r^{-1}$ in
every component. Componentwise CRT gives a unit $\beta\in G$.
Keep it precisely when
\begin{equation}\label{cy24:profile-test}
 T_8(2^{r_1}u_1,2^{r_2}u_2,
            (1-i)^{r_4}u_4,(1-\zeta_8)^{r_8}u_8)
                  \equiv0\pmod8.
\end{equation}
All operations are independent of the chosen lifts: multiplication by
$S_r$ takes a difference in $J$ back into $J$. The values $h(\beta)$
are precisely the profile image. This prescription provides compatible
representatives directly, without a list of machine-chosen anchors.

\begin{table}[htbp]
\centering
\caption{All nonempty two-profile images at valuations five through seven.
Every unlisted positive profile of the indicated total, with
$r_1+r_2+r_4\ge4$, has empty image.}
\label{cy24:profiles}
\begin{tabular}{ccrl}
\toprule
$k$ & $(r_1,r_2,r_4,r_8)$ & Kept tuples in \eqref{cy24:1024} & Classes\\
\midrule
5&$(1,2,1,1)$&256&$\{010\}$\\
5&$(2,1,1,1)$&256&$\{012\}$\\
6&$(1,3,1,1)$&256&$\{010\}$\\
6&$(3,1,1,1)$&256&$\{012\}$\\
7&$(1,1,2,3)$&512&$\{111,113\}$\\
7&$(1,1,3,2)$&512&$\{101,103\}$\\
7&$(1,4,1,1)$&256&$\{010\}$\\
7&$(4,1,1,1)$&256&$\{012\}$\\
\bottomrule
\end{tabular}
\end{table}

Table~\ref{cy24:profiles} is obtained by the finite test
\eqref{cy24:profile-test}, over the $3$, $9$ and $19$ possible
profiles. To reproduce the calculation, range $u_1,u_2$ over
$1,3,5,7$, and range $u_4,u_8$ over the eight type-$B$ normal forms
in Table~\ref{cy24:map}. Evaluate the eight expressions in
\eqref{cy24:T8}, keep exactly those divisible by eight, and compute
their classes by CRT and Table~\ref{cy24:map}. This exhausts $1024$
tuples per profile; the local normalization in the proof of
Proposition~\ref{cy24:cofactor} is what makes this finite list cover
arbitrary integral cofactors. Taking unions gives $B_5=B_6=B$ and
$B_7=B'$.

The seed for the odd two-exponent tail is
\begin{equation}\label{cy24:even-anchor}
 F_9(X)=2-X+X^2-X^3+X^4+X^9+X^{15}
                         +X^{20}-X^{21}+X^{22}-X^{23}.
\end{equation}
Its component norm vector in order $\mathcal D$ is
$(4,8,1,4,1,4,1,1)$, with $F_9(-1)=8$. Hence its determinant is
$512$, as used in \eqref{cy24:even-tail}. 


\section{Finite residue data for \texorpdfstring{$\C{2}\times\C{12}$}{C2 x C12}}
\label{c2c12:data}

We give exact finite definitions and generation rules for the local
class sets in Section~\ref{c2c12:section}, independently of numerical
encodings of $H$. Section~\ref{sec:computation} and the supplementary
computation index give the regeneration commands and certificates.

\subsection{Inverse evaluation and conductor units}
\label{c2c12:inverse-data}

For each fixed sign of $A$, write the tensor components as
\[
 a,b\in\Z,\quad u+iv\in\OO_4,\quad
 \beta,\delta\in\Z[\omega],\quad r+is\in\Z[\omega,i].
\]
The $d=6$ component $c+dt_6$ gives $\delta=c-d\omega$.
For the quartic component $c_0+c_1t+c_2t^2+c_3t^3$, the conventions
in~\eqref{c2c12:tensor} give
\[
 r=(c_0+c_2)+c_2\omega,\qquad s=-c_3+c_1\omega.
\]
Apply the $\C{2}\times\C{4}$ inverse below separately over $\Q$
and $\Q(\omega)$:
\begin{equation}
 \mathcal I_d=\frac18(n_++n_-,\ n_+-n_-),
 \label{c2c12:full-numerators}
\end{equation}
where $n_\pm$ are given by~\eqref{c2c12:local-inverse} and the
coefficient order is $(1,x,x^2,x^3,u,ux,ux^2,ux^3)$.
Write the outputs as $f_1$ and $f_2=f_{20}+f_{21}\omega$.
The common inverse~\eqref{eq:c3-inverse} gives
\begin{equation}
 q=\frac{f_1-f_{20}-f_{21}}3,\qquad
 F=(f_{20}+q)+(f_{21}+q)z+qz^2.
 \label{c2c12:full-inverse}
\end{equation}
Substitute $u=A$, $x=B^9$, $z=B^4$ to recover the original $24$
coefficients. Their integrality is the exact membership test for
$\mathcal A$.

The conductor and index formula~\eqref{eq:abelian-conductor} gives
Table~\ref{c2c12:conductor-table} and
\[
 |\operatorname{disc}(\mathcal M)|=2^{12}3^8,
 \qquad [\mathcal M:\mathcal A]=2^{30}3^8.
\]

Here are finite generators for the group $L$ in
\eqref{c2c12:quotient}.  At $2$, evaluate $u,x,z$ from
\eqref{c2c12:tensor} in the component residue rings.  Put
\[
 \Theta=\{1+z+z^2,\ 2-z-z^2,\ z(2-z-z^2)\}.
\]
The two-primary subgroup $L_2$ is generated by $z$ and
\begin{equation}
 \begin{gathered}
 1+\theta\,2^c(u-1)^a(x-1)^b,\\
 \theta\in\Theta,\quad 0\le c,a<3,\quad0\le b<4,\quad c+a+b>0.
 \end{gathered}
 \label{c2c12:L2-generators}
\end{equation}
Give these generators residue one at $3$, omitting identities.
For $I=(2,u-1,x-1)$, direct evaluation gives $I^8=0$ and quotient
$\F_2\times\F_4$, whose units are generated by $z$.
The elements of $\Theta$ span the cubic coefficient summands over
$\Z/8\Z$; their displayed monomials therefore span every $I$-layer.
The lifting argument~\eqref{eq:unit-filtration} proves completeness.
The integral image has $2^{30}$ elements, with three of eight
residue classes invertible, so $|L_2|=3\cdot2^{27}$.

At $3$, the integral image is the paired diagonal algebra
$\F_3^4\times\F_9^2$ in the pairs~\eqref{c2c12:three-pairs}.
For each of the four scalar pairs use the diagonal generator $2$;
for each of the two quadratic pairs use the diagonal generator $1+i$.
Give these six generators residue one at $2$.  They generate $L_3$,
of order $2^{10}$, and $L=L_2L_3$.

The generators of $E_+$ are equally explicit.  Pair the sign in the
distinguished rational component with each of the other three rational
signs.  In each Eisenstein component use a primitive sixth root; in
each Gaussian component use $i$; and in each quartic component use
$t$ and $1+t$. These are full by Lemma~\ref{lem:fields}, since
$(1+t)(t-1)=\omega$. The $13$ generators have residue image of
order $2^{21}3^6$.

For a canonical finite presentation, enumerate each residue-unit group
lexicographically in its power basis. Adjoin the first unit outside the
current subgroup, with the relation given by its least positive power
in that subgroup, until all units occur. Take the direct product of
these presentations and impose the displayed global and local unit
words. The rings have at most $256$ elements and their unit groups
at most $192$. Integer reduction gives nonunit diagonal entries
\begin{equation}
 2,2,2,2,2,4,4.
 \label{c2c12:invariant-factors}
\end{equation}
This proves Lemma~\ref{c2c12:quotient-lemma}; its relation words supply
the global correction in Lemma~\ref{lem:unit-correction}.

\subsection{Complete local fibres and the minima}
\label{c2c12:two-data}

At $2$, the decomposition by $z$ gives two copies of
$\C{2}\times\C{4}$, over $R_1=\Z_2$ and
$R_2=\Z_2[\omega]$, respectively.  In either copy order the components as
\begin{equation}
 (a_+,b_+,a_-,b_-,g_+,g_-),\qquad
 g_\pm=r_\pm+it_\pm.
 \label{c2c12:local-order}
\end{equation}
Put $\pi=1-i$ and $c=(3,3,3,3,4,4)$.  The conductor quotient of this
block is
\[
 W_d=(R_d/8R_d)^4\times(R_d[i]/4R_d[i])^2.
\]
Its integral image $\Lambda_d$ consists exactly of the residues for
which all coordinates of $n_++n_-$ and $n_+-n_-$ are divisible by $8$,
where
\begin{equation}
 n_\pm=(a_\pm+b_\pm+2r_\pm,\ a_\pm-b_\pm+2t_\pm,
          a_\pm+b_\pm-2r_\pm,\ a_\pm-b_\pm-2t_\pm).
 \label{c2c12:local-inverse}
\end{equation}
This condition is independent of the chosen lifts.

For $p\in\Z_{\ge0}^6$, put
$b_p=(2^{p_1},\ldots,2^{p_4},\pi^{p_5},\pi^{p_6})$.
Let $B_{d,p}$ be the full component tuple with $b_p$ in block $d$
and ones in the other block.  Its norm is $2^{d|p|}$.
Let $\mathcal A_{d,\bar p}$ be the set of classes
$h(\operatorname{CRT}(v,1))$ as $v\in W_d^\times$ ranges over
\begin{equation}
 b_pv\in\Lambda_d,
 \qquad \bar p_j=\min(p_j,c_j).
 \label{c2c12:fibres}
\end{equation}
Here and below, a block CRT tuple is the identity in the inactive block;
the two arguments specify its residues at $2$ and $3$.  The set in
\eqref{c2c12:fibres} depends only on $\bar p$.  Define
\begin{equation}
 \mathcal D_{d,p}=\mathcal A_{d,\bar p}
       +h\bigl(\operatorname{CRT}(1,B_{d,p}^{-1})\bigr).
 \label{c2c12:fixed-two-class}
\end{equation}
The inverse in this formula is taken only at $3$, where every entry of
$B_{d,p}$ is a unit.

The inverse congruences give the complete fibres. An inactive block
is an order unit because its augmentation is invertible modulo $2$.
Local order units normalize it, and every nonzero paired residue at $3$,
to one. This explains the opposite-prime shift
in~\eqref{c2c12:fixed-two-class}.

For a saturated coordinate the class sets have period two:
\begin{equation}
 \mathcal D_{d,p+2e_j}=\mathcal D_{d,p}\qquad(p_j\ge c_j).
 \label{c2c12:period-two}
\end{equation}
Indeed, the raw coordinate remains zero modulo its conductor at $2$.
For $j\le4$, its multiplier $4$ is one at $3$, so the cofactor is
unchanged. For $j=5,6$, multiply that cofactor component by the global norm-one
unit $i^{-1}$. At $3$ the corrected multiplier is
$\pi^2i^{-1}=-2=1$, and the cofactor class is unchanged.
The inverse operations prove equality. Thus this is a period of
classes modulo global units, not a period of the raw residue.

Put
\[
 \mathcal V_2=\prod_{j=1}^6\{0,\ldots,c_j+1\},
 \qquad \eta(p)=\begin{cases}1,&p_j\ge c_j\text{ for some }j,\\
                             0,&\text{otherwise}.
             \end{cases}
\]
For $n\ge0$, define
\begin{equation}
 \mathcal B_d(n)=
 \bigcup_{\substack{p\in\mathcal V_2\\
    n-|p|\in2\eta(p)\Z_{\ge0}}}
 \mathcal D_{d,p}.
 \label{c2c12:block-sets}
\end{equation}
Reduce each saturated coordinate to $c_j$ or $c_j+1$ using
\eqref{c2c12:period-two}; conversely, it absorbs every even excess.
When $\eta(p)=0$, no padding is allowed and $n=|p|$.
Write $R_n=\mathcal B_1(n)$ and $Q_n=\mathcal B_2(n)$.
An empty union is empty.  The quadratic index $n$ contributes $2n$ to
the valuation of the rational determinant.

The local inverse~\eqref{c2c12:local-inverse} has denominator $8$ and
evaluation determinant of absolute value $2^{10}$.  Its conductor is
$(8R_d)^4\times(4R_d[i])^2$.  The integral image therefore has
$2^{10d}$ residues.  A complete set for $d=1$ is obtained by evaluating
\begin{equation}
 \sum_{j=0}^7 a_j e_j,\qquad
 (e_0,\ldots,e_7)=(1,x,x^2,x^3,u,ux,ux^2,ux^3),
 \label{c2c12:residue-box}
\end{equation}
with
\[
 0\le a_0<8,\quad0\le a_1,a_4<4,\quad
 0\le a_2,a_3,a_5<2,\quad a_6=a_7=0.
\]
If the difference of two such evaluations is in the conductor, the
scalar and Gaussian congruences successively force the differences
$a_5,a_4,a_3,a_2,a_1,a_0$ to vanish.  The $1,024$ evaluations are
therefore distinct and exhaust the image.  For $d=2$, take independent
representatives for the coefficients of $1$ and $\omega$; this gives
all $1,048,576$ residues.

For every resulting residue $r$, record its clipped valuations
$\bar p_j=\min(v_{\varpi_j}(r_j),c_j)$, using
$\varpi_j=2$ in the first four positions and $\varpi_j=\pi$ in the last
two.  Set the clipped valuation of zero equal to $c_j$.
To recover every cofactor residue, enumerate the units $v_j$ of the
individual component quotient satisfying
\begin{equation}
 \varpi_j^{\bar p_j}v_j=r_j.
 \label{c2c12:fibre-equation}
\end{equation}
If $v_{j,0}$ is one solution, all solutions are exactly
\[
 v_{j,0}K_{j,\bar p_j},\qquad
 K_{j,e}=\{w\in W_{d,j}^\times:
                \varpi_j^e(w-1)=0\}.
\]
Thus a single product of solutions gives a starting class, and the
images of the six groups $K_{j,e}$ give every remaining class in that
fibre.  Taking their union over the residue box computes
$\mathcal A_{d,\bar p}$ exactly. Each component unit enumeration has
size at most $192$.

The primitive total is at most $26$. For $h\in H$ and $s\in\Z/2\Z$,
retain
\begin{equation}
 m_d(h,s)=\min\{|p|:p\in\mathcal V_2,\ \eta(p)=1,
              \ |p|\equiv s\pmod2,\ h\in\mathcal D_{d,p}\}.
 \label{c2c12:block-minimum}
\end{equation}
Also retain each exact pair $(h,|p|)$ with $\eta(p)=0$.
An empty minimum is infinity. Table~\ref{c2c12:block-records} gives
the complete record counts.

\begin{table}[ht]
\centering
\caption{Complete local-two records.  An exact record is a pair of a
class and a total; a padded record is a finite minimum in
\eqref{c2c12:block-minimum}.}
\label{c2c12:block-records}
\begin{tabular}{c r r r r}
\toprule
Base degree & clipped profiles & exact records & padded records & largest minimum\\
\midrule
$1$ & $102$ & $45$ & $64$ & $13$\\
$2$ & $149$ & $889$ & $256$ & $10$\\
\bottomrule
\end{tabular}
\end{table}

To combine the blocks with scalar period $4$, use each rational
padded record at its minimum cost and that cost plus $2$.
Use each quadratic padded record at its minimum relative cost.
Combine all pairs, adding their classes and taking total cost $a+2r$.
Two exact records give a rigid record of cost at most $42$.
If at least one record is padded, retain the least cost for its class
and total residue modulo $4$, denoted $m(h,s)$.
The extra rational $2$ represents both residues modulo $4$;
all further increments $4$ can be absorbed in a saturated coordinate,
adding $4$ in the rational block or $2$ in the quadratic block.
Thus, with $E_k$ the rigid exact set,
\begin{equation}
 T_k=E_k\cup\{h\in H:m(h,k\bmod4)\le k\}.
 \label{c2c12:combined-minimum}
\end{equation}
The $198,085$ primitive pairs give all $2,048$ class-and-remainder
minima; their largest value is $26$.

\begin{table}[ht]
\centering
\caption{The initial two-primary class sets.  The full sets, defined
by~\eqref{c2c12:combined-minimum}, are used in the membership test.}
\label{c2c12:initial-two}
\begin{tabular}{r r@{\qquad}r r@{\qquad}r r}
\toprule
$k$ & $|T_k|$ & $k$ & $|T_k|$ & $k$ & $|T_k|$\\
\midrule
$0$ & $1$ & $12$ & $85$ & $18$ & $152$\\
$1$--$7$ & $0$ & $13$ & $32$ & $19$ & $32$\\
$8$ & $12$ & $14$ & $92$ & $20$ & $332$\\
$9$ & $8$ & $15$ & $32$ & $21$ & $272$\\
$10$ & $8$ & $16$ & $152$ & $22$ & $512$\\
$11$ & $24$ & $17$ & $32$ & &\\
\bottomrule
\end{tabular}
\end{table}

If $k\ge23$ and $m(h,k\bmod4)>k$, congruence forces
$m(h,k\bmod4)\ge k+4>26$, a contradiction.
Together with $T_{22}=H$, this proves the full tail.
Testing the two classes $0$ and
$h_{j_0}(-1)$ in the same finite sets gives, for $k\le21$, the positive
exponents $0,8,12,14,16,18,20,21$ and the negative exponents
$13,15,17,19,20,21$.  These are the finite checks used in
Corollary~\ref{cor:pure-two}.

\subsection{The three-primary recurrence and reconstruction}
\label{c2c12:three-data}

At $3$, pair the components in the order
\begin{equation}
 \begin{aligned}
 &((0,1),(0,3)),\quad ((0,2),(0,6)),\quad ((0,4),(0,12)),\\
 &((1,1),(1,3)),\quad ((1,2),(1,6)),\quad ((1,4),(1,12)).
 \end{aligned}
 \label{c2c12:three-pairs}
\end{equation}
Their residue degrees are $w=(1,1,2,1,1,2)$.  In a pair of degree $w_j$,
use the raw factors $(3^{x_j},\lambda^{y_j})$, with
$\lambda=1-\omega$.  Their norm contribution is
$3^{w_j(x_j+y_j)}$.  A pair is either empty, with $x_j=y_j=0$, or
occupied, with $x_j,y_j\ge1$.  For an occupied pair choose a residue
ratio $g_j^{z_j}$, with $g_j=2$ in $\F_3$ and $g_j=1+i$ in $\F_9$.
Let $U(x,y,z)\in T$ have residues
\begin{equation}
 \begin{array}{c|cc}
 &\text{at }2&\text{at }3\\ \hline
 \text{lower entry}&3^{-x_j}&1\\
 \text{upper entry}&\lambda^{-y_j}&g_j^{z_j}
 \end{array}
 \label{c2c12:three-anchor}
\end{equation}
at occupied pairs, and entries one at empty pairs.  The inverse factors
in this formula are taken only at $2$.

Retain just the following finite options at each occupied position:
\begin{equation}
 \begin{cases}
 1\le x_j\le4,\quad1\le y_j\le4,\quad0\le z_j<2,&w_j=1,\\
 1\le x_j\le2,\quad1\le y_j\le4,\quad0\le z_j<8,&w_j=2.
 \end{cases}
 \label{c2c12:three-options}
\end{equation}
For each nonempty retained word put $\beta(x,y)=\sum_jw_j(x_j+y_j)$,
and set $S_0=\{0\}$.  For $b>0$, define
\begin{equation}
 S_b=\left\{h(U(x,y,z)):
 \begin{array}{l}
 (x,y,z)\text{ is a nonempty retained word},\\[-2pt]
 \beta(x,y)\le b,\quad\beta(x,y)\equiv b\pmod4
 \end{array}\right\}.
 \label{c2c12:three-set}
\end{equation}

At $3$, the two residues in each pair must agree.  They are both zero
precisely at occupied pairs.  At such a pair the unit cofactors have an
arbitrary nonzero ratio, whereas a paired local unit normalizes the
lower residue to one.  At $2$, the inverse raw factors in
\eqref{c2c12:three-anchor} cancel the three-primary factors.
The lower inverse has period $2$.  The upper class has period $4$,
because $\lambda^4=9\omega^2$: the factor $9$ is one in both
two-primary conductor rings, and $\omega$ is a global norm-one unit
whose upper residue at $3$ is one.  Thus every support reduces to
\eqref{c2c12:three-options}, preserving its class and its cost modulo
$4$.  Every nonempty retained support absorbs a further cost $4$ by
increasing one lower exponent by $4/w_j$.  This proves
\eqref{c2c12:three-set}.  An occupied degree-one pair realizes every
$b\ge2$; no occupied pair has cost one.

Combine the options~\eqref{c2c12:three-options} by the following
minimum recurrence.  At each position
include the empty option, with class zero and cost zero, as well as
the $32$ or $64$ occupied options.  Retain a minimum for each triple
$(h,r,\epsilon)\in H\times\Z/4\Z\times\{0,1\}$, where $\epsilon$
records whether a pair has been occupied.  Start with value zero at
$(0,0,0)$ and infinity elsewhere, and initialize each new stage to
infinity.  If an option has class $a$, cost
$c$, and occupied flag $\delta$, the update is
\begin{equation}
 M_j(h+a,r+c,\max(\epsilon,\delta))
  =\min\bigl\{M_j(h+a,r+c,\max(\epsilon,\delta)),
              M_{j-1}(h,r,\epsilon)+c\bigr\},
 \label{c2c12:three-recurrence}
\end{equation}
with the second coordinate reduced modulo $4$.  This recurrence is
just the minimum over all words ending in the chosen option.
It is proved by induction on the six positions.  A primitive word has
cost at most $56$, and every nonempty word admits the padding just
proved.  Consequently, for $b>0$,
\[
 S_b=\{h:M_6(h,b\bmod4,1)\le b\}.
\]
The two-primary image splits into the degree-one and degree-two
blocks, giving the normalized classes $R_a+Q_r$; the paired analysis
above gives $S_b$. The common CRT identity~\eqref{eq:crt-assembly},
followed by the union over $a+2r=k$, proves~\eqref{c2c12:allowed}.
A retained word and the stated padding operations give the raw factors
and a residue anchor. Lemma~\ref{lem:unit-correction} corrects the
supplied cofactor tuple, and~\eqref{c2c12:full-inverse} recovers its
$24$ coefficients; the anchor does not replace the supplied norms.


\section{Finite data and anchors for \texorpdfstring{$\C{2}\times\C{2}\times\C{6}$}{C2 x C2 x C6}}
\label{c2c2c6:data}

Character positions remain ordered as in Section~\ref{c2c2c6:section}.
Section~\ref{sec:computation} indexes the complete regenerations.

\subsection{The conductor quotient}
\label{c2c2c6:quotient-data}

We give generators for the local order-unit subgroup $L$.
Put $\delta_1=u-1$, $\delta_2=v-1$, and $\delta_3=w-1$.
At two, the integral order image $A_2$ in the component rings modulo eight is
\begin{equation}\label{c2c2c6:local-ring-presentation}
 \frac{(\Z/8\Z)[z,\delta_1,\delta_2,\delta_3]}
 {\bigl(z^3-1,\ \delta_i^2+2\delta_i,\ 4\delta_i,
       \ 2\delta_i\delta_j\ (i<j),\ \delta_1\delta_2\delta_3\bigr)}.
\end{equation}
To check this presentation, evaluate each displayed relation at $u,v,w=\pm1$.
For each of $1,z,z^2$, a normal form has one coefficient modulo eight, three modulo four, and three modulo two.
The source therefore has order $2^{36}$.
The Walsh transform has determinant of absolute value $2^{12}$ in each of the three coefficient positions.
Hence its image modulo eight also has order $2^{36}$, proving that the presentation is exact.

The ideal $I=(2,\delta_1,\delta_2,\delta_3)$ satisfies $I^3=0$ and
$A_2/I\simeq\F_2\times\F_4$.
The classes of
\[
 2z^j,\ \delta_i z^j\quad(0\le j<3,\ 1\le i\le3)
\]
give a basis of $I/I^2$, and those of
\[
 \begin{gathered}
 4z^j\quad(0\le j<3),\qquad
 2\delta_i z^j\quad(0\le j<3,\ 1\le i\le3),\\
 \delta_i\delta_l z^j\quad(0\le j<3,\ 1\le i<l\le3)
 \end{gathered}
\]
give the remaining $21$ basis elements of $I^2$.
These layer bases and \eqref{eq:unit-filtration} show that $A_2^\times$
has order $3\cdot2^{33}$ and is generated by $z$ and the $33$ elements
$1+3b$ from the displayed monomials.
The listed polynomials have three-primary conductor residue one.

At three, the order image is the diagonal paired copy of $\F_3^8$.
Its units are generated by changing the sign in a single pair.
An explicit polynomial for this change is $1+8J_h$, where
\[
 J_h=\sum_gW_{hg}u^{g_0}v^{g_1}w^{g_2}.
\]
It is one modulo eight and is minus one precisely at the selected pair modulo three.
The two local choices combine independently by the Chinese remainder theorem.
Thus
\begin{equation}\label{c2c2c6:group-orders}
 |\Gamma|=8^8\,96^8=2^{64}3^8,\qquad
 |E^+|=2^7 6^8=2^{15}3^8,\qquad |L|=3\cdot2^{41}.
\end{equation}
An element of $E^+\cap L$ lifts to an integral tuple of roots of unity.
If its group-ring coefficients are $a_g$, Fourier orthogonality gives
\[
 \sum_g|a_g|^2=\frac1{24}\sum_\chi|\chi(F)|^2=1.
\]
Thus $F=\pm g$ for one group element.
Conversely, these $48$ elements belong to $E^+\cap L$: a translation has sign $(-1)^{24-24/d}=1$ for $d=1,2,3,6$, and coefficient negation also has determinant one.
It follows from \eqref{c2c2c6:group-orders} that
\[
 |H|=\frac{|\Gamma|\,48}{|E^+|\,|L|}=4096.
\]
The seven paired rational sign changes and the eight individual copies of $1+\omega$ generate $E^+$.
Together with the $42$ generators of $L$ above, they give a presentation that both forms the quotient and recovers a unit correction.
A concrete ambient presentation uses $3,5$ modulo eight and $2$ modulo
three at each rational position. At each Eisenstein position use
\[
 u_1=1+3\omega,\qquad u_2=1+7\omega,\qquad u_3=2+\omega
 \quad\text{modulo }8O,
\]
with relations $u_1^{12}=1$, $u_2^2=u_1^2$, $u_3^2=u_1^{10}$,
and use $2$ modulo $\lambda$.
The $12\cdot2\cdot2=48$ normal forms exhaust $(O/8O)^\times$.
There are $56$ ambient generators. Adjoining the $15$ generators of
$E^+$ and the $42$ generators of $L$ gives an integer relation matrix
with $56$ rows and $113$ columns; denote this matrix by $M$.
Its Smith diagonal is
\begin{equation}
 \underbrace{1,\ldots,1}_{45},\quad
 \underbrace{2,\ldots,2}_{10},\quad4,
 \qquad H\simeq\C{2}^{10}\times\C{4}.
 \label{c2c2c6:invariant-factors}
\end{equation}
The certificate \texttt{rebuilt/c2c2c6/invariants.json}, regenerated by
\texttt{check\_c2c2c6\_invariants.py} in the supplement, supplies
unimodular matrices $U,V$ with $UMV$ equal to the displayed Smith form.
Every added generator of $E^+$ or $L$ comes from an actual global
component unit or integral group-ring polynomial, so a relation word
supplies its own lift.

\paragraph{The small Eisenstein quotients.}
In $Q_8$, $[1-\omega]$ and $[1+4\omega]$ have orders $4$ and $2$
and norm residues $3$ and $5$ modulo eight, so they are independent.
The $48$ ambient units and six distinct global-unit residues give
$Q_8\simeq\C{4}\times\C{2}$.
By CRT, $Q_{16}\to Q_8$ has order-two kernel $\langle[17]\rangle$.
The lifts $17-\omega$ and $9+4\omega$ have residue one modulo $\lambda$;
their fourth and second powers in the CRT ring are $\omega^2$ and $1$.
They give a complement to the kernel, proving $Q_{16}\simeq\C{4}\times\C{2}^2$.

\subsection{Local valuations and profile sets}
\label{c2c2c6:profile-rules}

\begin{lemma}\label{c2c2c6:local-valuations}
Let $R$ be $\Z_2$ or the unramified quadratic ring $\Z_2[\omega]$, and let its residue field be $k$.
For eight nonzero values $g_h$ satisfying $Wg\in8R^8$, the possible totals $\sum_hv_2(g_h)$ are
\[
 \{0,8\}\cup\Z_{\ge12}\quad(k=\F_2),\qquad
 \{0,8\}\cup\Z_{\ge10}\quad(k=\F_4).
\]
At rational total eight, the product divided by $2^8$ is one modulo four.
\end{lemma}

\begin{proof}
All values have the same residue modulo two.
Thus either all are units or all are even.
For $g_h=2b_h$, the exact condition is
\begin{equation}\label{c2c2c6:affine-test}
 Wg\in8R^8
 \quad\Longleftrightarrow\quad
 \overline b\colon\F_2^3\longrightarrow k\text{ is affine},
 \qquad\sum_hb_h\equiv0\pmod{4R}.
\end{equation}
Subtracting each Walsh row from the row indexed by zero shows that the residue vector is orthogonal to $1,h_0,h_1,h_2$.
These four functions span their own orthogonal complement: their pairwise products have even sums, and their span has dimension four.
This proves \eqref{c2c2c6:affine-test}.

For a positive valuation vector $e$, the zero set of this affine map is $Z_e=\{h:e_h\ge2\}$.
A proper zero set is empty or an affine plane over $\F_2$; over $\F_4$ it may also be an affine line.
Conversely, choose an affine map with the prescribed zero set and lift its nonzero values to the positions with $e_h=1$.
At every other position take $b_h=2^{e_h-1}$.
The sum of these lifts is even; subtract it from one unit position.
This preserves its valuation and gives the required sum modulo four.
If $Z_e$ is all eight positions, the sole condition is that the nonzero residues at $e_h=2$ sum to zero.
Over $\F_2$ their number must be even; over $\F_4$ every number except one is possible, using equal pairs or $1,\omega,\omega^2$ followed by equal pairs.
These arguments prove every profile condition, including arbitrary exponents.

The smallest positive total is eight.
A nonempty proper zero set has at least four positions over $\F_2$ and two over $\F_4$, giving the next lower bounds twelve and ten.
All larger rational totals are supplied by
\[
 (4,4,4,4,2,2,2,2),\qquad
 (-2^{a-10},4,4,4,-2,2,2,2)\quad(a\ge13).
\]
The quadratic totals ten and $r\ge11$ are supplied by
\[
 (4,4,2,2,2\omega,2\omega,2\omega^2,2\omega^2),\qquad
 (-2^{r-8},4,-2,2,2\omega,2\omega,2\omega^2,2\omega^2).
\]
Their integrality follows from \eqref{c2c2c6:affine-test}.
Finally, at rational total eight, write $b_h=1+2d_h$.
The sum condition is $\sum_hd_h=0$ modulo two, which is equivalent to $\prod_hb_h=1$ modulo four.
\end{proof}

\paragraph{Rational classes as affine cosets.}
Put
\[
 \alpha=\theta_{0,1}(13),\quad \delta=\theta_{0,1}(11),\quad
 c_h=\theta_{h,3}(17),\quad c=c_0.
\]
The quotient presentation gives six independent order-two columns
$\alpha,\delta,c,L(e_0),L(e_1),L(e_2)$, where
\begin{equation}\label{c2c2c6:rational-component-relations}
 \begin{gathered}
 \theta_{h,1}(13)=\alpha,\quad \theta_{h,1}(11)=\delta,\quad
 \theta_{h,1}(17)=c_h,\\
 \theta_{h,1}(19)=c_h+\delta,\qquad c_h=c+L(h),
 \end{gathered}
\end{equation}
and $L$ is linear.
For the first two identities, use the integral tuples with $13$, or $-11$,
in two rational positions and one elsewhere.
The paired $17$ tuple is integral, and $17\cdot11\equiv19\pmod{24}$.
The rational tuple equal to $-17$ on an affine plane $A$ and one
elsewhere has Walsh increment $-18W1_A\in8\Z^8$ and three-residues one.
Removing its four signs proves $\sum_{h\in A}c_h=0$ and hence affinity.
Set
\[
 V=L(\F_2^3),\quad U=\langle\alpha,V\rangle,\quad
 K=\langle\alpha,c+\delta,V\rangle.
\]
Thus $|V|=8$, $|U|=16$, $|K|=32$, and $c\notin K$.

\begin{proposition}\label{c2c2c6:rational-cosets}
The complete rational class sets are
\begin{equation}\label{c2c2c6:rational-coset-formula}
 R_a=
 \begin{cases}
 \{0\},&a=0,\\
 U,&a=8,\\
 (c+K)\setminus(\delta+\langle\alpha\rangle),&a=13,\\
 K,&a\ge12\text{ even},\\
 c+K,&a\ge15\text{ odd},\\
 \varnothing,&\text{otherwise}.
 \end{cases}
\end{equation}
\end{proposition}

\begin{proof}
For a positive profile $e$, write $u_h=13^{x_h}19^{y_h}$ with $x_h,y_h\in\F_2$.
With $A=\{h:e_h=1\}$ and $t=\#\{h:e_h=2\}$,
\eqref{c2c2c6:affine-test} becomes
\begin{equation}\label{c2c2c6:rational-parity}
 \sum_{h\in A}y_h=t\pmod2.
\end{equation}
Set $E=\sum_he_h\bmod2$, $X=\sum_hx_h$, $Y=\sum_hy_h$,
$T=\sum_hy_hh$, $r=\sum_h(e_h\bmod2)h$, and $Z=T+r$.
The normalized class is
\begin{equation}\label{c2c2c6:rational-moment-class}
 Ec+X\alpha+Y(c+\delta)+L(Z).
\end{equation}
The four moment columns at $0,e_0,e_1,e_2$ are independent, so $(Y,T)$
is arbitrary; $X$ is free.
Even totals therefore lie in $K$ and odd totals in $c+K$.
At total eight, $Y=r=0$, giving $U$.

For $A=\{\ell(h)=s\}$, $\ell\ne0$, the parity equation reads
$(1+s)Y+\ell(Z)=t+\ell(r)$.
The outside exponent patterns
\[
\begin{array}{c|cccc}
a&12&13&14&15\\ \hline
(e_h)_{h\notin A}&(2,2,2,2)&(3,2,2,2)&(3,3,2,2)&(3,4,2,2)\\
t+\ell(r)&0&s&0&1+s
\end{array}
\]
realize every moment at totals $12,14,15$: choose $s=1$ and a nonzero
$\ell$ annihilating $Z$.
At total thirteen the displayed pattern is forced.
Its plane union omits precisely $(Y,Z)=(1,0)$, hence the two stated classes.
The total-fourteen and total-fifteen profiles contain an exponent at least
three; any even increase preserves its zero residue modulo eight and
normalizing parity at three.
Lemma~\ref{c2c2c6:local-valuations} settles the other totals.
\end{proof}

Solving the moments supplies an anchor.
Lemma~\ref{lem:unit-correction} then lifts the prescribed cofactors,
preserving their absolute rational norms, quadratic norms and total sign.
\paragraph{Quadratic profile classes.}
For a quadratic profile $f$ of total at most fifteen, put $L_f=\{h:f_h=1\}$ and $t_f=\#\{h:f_h=2\}$.
The set $L_f$ is nonempty.
Among the $256$ affine maps $c\colon\F_2^3\to\F_4$, retain those whose support is $L_f$.
For each map, use the Teichm\"uller lift $\tau(c_h)\in\mu_3$ at its nonzero positions.
At positions with $f_h=2$, choose nonzero residues $d_h\in\F_4$
and use the phase $\tau(d_h)$; at positions with $f_h\ge3$, use phase one.
Writing the principal parts of the cofactors modulo four as $1+2v_h$, the single remaining equation is
\begin{equation}\label{c2c2c6:quadratic-linear-equation}
 \sum_{h\in L_f}c_hv_h=\sum_{f_h=2}d_h\quad\text{in }\F_4.
\end{equation}
To obtain it, the sum of the low Teichm\"uller lifts is zero modulo four: each nonzero affine value occurs four or eight times, or all three nonzero values occur twice.
The right side of \eqref{c2c2c6:quadratic-linear-equation} ranges over $\{0\}$, $\F_4^\times$, or $\F_4$, according as $t_f$ is zero, one, or at least two.

The full inverse image of this binary linear equation is essential.
In each component lift the binary basis of $\F_4$ by $3$ and $1+2\omega$, and generate the kernel of reduction modulo four by $5$ and $1+4\omega$.
Give each lift three-residue one by the Chinese remainder theorem.
Choose a pivot in $L_f$.
Solving for its two binary coordinates gives a basis of fourteen vectors for the homogeneous equation on $\F_4^8$.
Lift these vectors componentwise and adjoin all sixteen component generators $5,1+4\omega$.
They generate the entire inverse image of the homogeneous equation, since
\[
 (1+2O/8O)^8/(1+4O/8O)^8\simeq(\F_4^8,+).
\]
For each allowed right side, multiply by one particular lifted solution.
The Teichm\"uller phases do not alter the $H$-class, since they are global units; they are retained to construct the raw anchor satisfying \eqref{c2c2c6:quadratic-linear-equation}.
Finally translate by $\sum_{f_h\text{ odd}}\theta_{h,3}(17)$.
This gives exactly $Q_f$, with a representative for every resulting class.

For the quadratic calculation, reduce each exponent at least five by two
until all entries lie in $\{1,2,3,4\}$. This preserves the raw residues
modulo eight, the positions of exponents one and two, and the exponent
parities. Conversely, an entry three or four absorbs any even excess;
a profile containing only one and two contributes only at its exact total.
Thus these $4^8$ primitive profiles give every required $Q_r$.

\subsection{Support shifts and assembly}
\label{c2c2c6:support-rule}

The identity $\lambda^4=9\omega^2$ implies $\eta_h(y+4)=\eta_h(y)$: its extra inverse residue modulo eight is the global unit $\omega$, whose residue modulo $\lambda$ is one.
Also $2d_h=2\kappa_h=0$.
For each of the eight positions, list the seventeen options
\[
 (x,y,z)=(0,0,0),\qquad
 x\in\{1,2\},\quad y\in\{1,2,3,4\},\quad z\in\{0,1\}.
\]
The nonempty options have cost $x+y$ and shift $(x\bmod2)d_h+\eta_h(y)+z\kappa_h$.
Successively add the eight lists.
For each state, cost parity, and flag indicating nonempty support, retain a word of least cost.
If $M(s,p)$ is the least cost for a nonempty final word, then
\begin{equation}\label{c2c2c6:S-minimum}
 S_b=\{s:M(s,b\bmod2)\le b\}\quad(b>0).
\end{equation}
The recurrence is exact by induction on the positions: appending the same option cannot make a larger previous cost preferable.
Reduction of the positive exponents proves necessity of \eqref{c2c2c6:S-minimum}.
For sufficiency, add the even difference to one positive $x$.
The empty word is treated separately and cannot be padded.
Every primitive word costs at most $48$, so this is a fixed finite calculation for all $b$.

\begin{proposition}[Anchor assembly]\label{c2c2c6:assembly}
Let $8\le k\le31$ be an allowed $2$-adic valuation and let $b\ge0$.
A supplied tuple $C=(u_h,\gamma_h)$ of integral components with norms prime to six can be completed by powers of $2$, $3$, and $\lambda$, and a unit in $E^+$, precisely when
$\theta(C)\in C_{k,b}$.
The resulting determinant is $2^k3^b\prod_hu_h\Norm(\gamma_h)$.
\end{proposition}

\begin{proof}
At two, the rational and quadratic order factors split independently.
At three, paired order units normalize the rational cofactor residues
to one; the field residues are then $2^{e_h-f_h}=2^{e_h+f_h}$ modulo
$\lambda$. Split this normalized tuple into a rational part with
three-residues $(1,2^{e_h})$ and a quadratic part with three-residues
$(1,2^{f_h})$, placing one in the other two-primary components.
Their classes are exactly $R_e$ and $Q_f$, respectively, and their
products give every compatible two-primary class. A block of total
zero is an order unit and contributes class zero.

Now choose a support $(x_h,y_h)$ and binary choices $z_h$.
If $C_0$ is an unramified cofactor anchor, define a new one by
\begin{equation}\label{c2c2c6:CRT-anchor}
\begin{array}{ll}
 u'_h\equiv3^{-x_h}(C_0)_{h,1}\pmod8,
   &u'_h\equiv1\pmod3,\\
 \gamma'_h\equiv\lambda^{-y_h}(C_0)_{h,3}\pmod{8O},
   &\gamma'_h\equiv2^{e_h-f_h}(-1)^{z_h}\pmod\lambda.
\end{array}
\end{equation}
Here $z_h=0$ at unoccupied positions.
The raw two-parts remain unchanged after multiplication by
$(2^{e_h}3^{x_h},2^{f_h}\lambda^{y_h})$.
The raw three-parts match at unoccupied positions and both vanish at occupied ones.
The class difference is exactly \eqref{c2c2c6:support-shift}.
Conversely, multiplication of the cofactor two-parts by $3^{x_h}$ and
$\lambda^{y_h}$, followed by resetting the field three-residue to
$2^{e_h-f_h}$, removes precisely this shift.
Thus the CRT identity \eqref{eq:crt-assembly}, with these profile and
support classes, gives \eqref{c2c2c6:C-kb}.
Lemma~\ref{lem:unit-correction} corrects the supplied tuple without
changing its signed norm product, and \eqref{c2c2c6:inverse} recovers
the coefficients.
\end{proof}

The complete class sets give
\begin{equation}\label{c2c2c6:full-target-identities}
\begin{aligned}
 R_8+Q_8+S_b&=H &&(b=3,4),\\
 R_a+Q_8+S_b&=H &&(12\le a\le15,\ 2\le b\le4).
\end{aligned}
\end{equation}
The supplement verifies all $4096$ classes in each sumset and retains
anchors for their constructive lifts.

\subsection{Constructions for the two tails}
\label{c2c2c6:tails}

We prove $2^{13}3^bm\in\Sdet{G}$ for $b\ge6$ and signed $m$ prime to six.
Fix $v\ne0$ in $\F_2^3$.
The identities $\lambda^2=-3\omega$ and $\lambda^4=9\omega^2$ give
$\eta_h(2)=\theta_{h,3}(13)$ and $\eta_h(4)=0$ after removing global-unit phases.
Putting $13$ in two field positions and one elsewhere gives an integral
tuple, so $\theta_{h,3}(13)$ is independent of $h$; it is killed by two
since $13^2\equiv1\pmod{8\lambda O}$.
Thus the following supports have the stated shifts:
\begin{equation}\label{c2c2c6:tail13-supports}
\begin{array}{c@{\quad}l@{\quad}l}
b&\text{nonzero }(h;x_h,y_h,z_h)&\text{shift}\\ \midrule
6&(0;1,2,0),\ (v;1,2,0)&L(v)\\
6&(v;2,4,1)&c+L(v)\\
7&(v;3,4,0)&c+\delta+L(v)\\
7&(0;1,2,0),\ (v;2,2,1)&\delta+L(v).
\end{array}
\end{equation}
Since $(\Z/24\Z)^\times=\langle13,11,17\rangle$, the scalar baseline
$c(m)$ lies in $\langle\alpha,\delta,c\rangle$ and has vector moment zero.
For either parity of $b$, choose the row whose $c$-coefficient modulo
$K$ is opposite to that of $c(m)$.
The residual class then has $E=1$ and $Z=v\ne0$, hence belongs to $R_{13}$.
Increase a positive $x_h$ by any even excess.
Proposition~\ref{c2c2c6:assembly} now realizes $2^{13}3^bm$ with
the prescribed signed scalar $m$.

For the high-two family, write $A=2^am$ with $m$ signed and odd, and define
\begin{equation}\label{c2c2c6:high-rational}
 R(A)=
 \begin{cases}
 (4m,4,4,4,2,2,2,2),&a=12,\\
 (-2^{a-10}m,4,4,4,-2,2,2,2),&a\ge13.
 \end{cases}
\end{equation}
Its product is $A$, and its Walsh transform is divisible by eight.
The residues at positions $1,\ldots,7$ are respectively
$(1,1,1,2,2,2,2)$ and $(1,1,1,1,2,2,2)$ modulo three.
Choose $\beta=2\gamma$ from
\begin{equation}\label{c2c2c6:high-quadratic}
\begin{array}{c@{\quad}c@{\quad}l}
 a&\text{support at three}&\gamma\\ \midrule
 12&\varnothing&(2\epsilon,2,-1,-1,\omega,\omega,\omega^2,\omega^2),\\
 a\ge13&\varnothing&(2\epsilon,2\omega,-1,-1,-\omega^2,\omega^2,\omega,\omega),\\
 12&\{0\}&(2\lambda,2\omega^2,-1,-1,\omega,\omega,\omega^2,\omega^2),\\
 a\ge13&\{0\}&(2\lambda,2,-1,-1,-\omega,\omega,\omega^2,\omega^2).
\end{array}
\end{equation}
In the first two rows $\epsilon\in\{1,-1\}$.
Each vector $2\gamma$ has Walsh transform divisible by eight.
Its norm product is $2^{20}$ in the unoccupied rows and $3\cdot2^{20}$ in the occupied rows.
At positions $1,\ldots,7$, its residues modulo $\lambda$ are those of \eqref{c2c2c6:high-rational}; its position-zero residue is $\epsilon$ or zero.

Given $D$ with $v_2(D)\ge32$ and $v_3(D)\ne1$, take
\[
 A=D/2^{20}\quad\text{if }3\nmid D,\qquad
 A=D/(3\cdot2^{20})\quad\text{if }9\mid D.
\]
Then $a\ge12$.
Use the matching row of \eqref{c2c2c6:high-quadratic}, choosing $\epsilon$ to match $R(A)_0$ when needed.
The tuple $(R(A),2\gamma)$ satisfies all the gluing congruences and has determinant $D$.

\subsection{The pure-two baseline classes}\label{c2c2c6:pure-two-data}

Put $c_-=c(-1)=\theta_{0,1}(-1)$.
The profile rules above give the finite intersections used in
Corollary~\ref{cor:pure-two}:
\[
\begin{array}{c@{\qquad}c}
 k&C_{k,0}\cap\{0,c_-\}\\ \midrule
 8,\ 12&\{0\}\\
 13&\varnothing\\
 14\le k\le31,\ k\text{ even}&\{0\}\\
 14\le k\le31,\ k\text{ odd}&\{c_-\}.
\end{array}
\]


\bibliographystyle{amsplain}
\bibliography{references}
\end{document}